\documentclass[11pt]{amsart}
\usepackage{amsmath,amssymb,amsthm,tikz}

\RequirePackage{color}
\RequirePackage[colorlinks, urlcolor=my-blue,linkcolor=my-red,citecolor=my-green]{hyperref}
\definecolor{my-blue}{rgb}{0.0,0.0,0.6}
\definecolor{my-red}{rgb}{0.5,0.0,0.0}
\definecolor{my-green}{rgb}{0.0,0.5,0.0}
\definecolor{nicos-red}{rgb}{0.75,0.0,0.0}
\definecolor{light-gray}{gray}{0.6}
\definecolor{really-light-gray}{gray}{0.8}
\definecolor{sussexg}{rgb}{0.0,0.5,0.5}
\definecolor{sussexp}{rgb}{0.5,0.0,0.5}

\usepackage{nicefrac}
\usepackage{graphicx}

\newtheorem{theorem}{\sc Theorem}[section]
\newtheorem{lemma}[theorem]{\sc Lemma}
\newtheorem{proposition}[theorem]{\sc Proposition}

\newtheorem{conjecture}[theorem]{\bf Conjecture}

\numberwithin{equation}{section}

\theoremstyle{remark}
\newtheorem{remark}[theorem]{Remark}
\newtheorem{example}[theorem]{\bf Example}

\newcommand{\be}{\begin{equation}}
\newcommand{\ee}{\end{equation}}

\def\bE{\mathbb{E}}
\def\bN{\mathbb{N}}
\def\bP{\mathbb{P}}

\def\bR{\mathbb{R}}
\def\bZ{\mathbb{Z}}

 \def\Z{\bZ}  \def\R{\bR}\def\N{\bN}

\def\e{\varepsilon}

\def\m1{\mathbf{1}}

\def\E{\bE}
\def\P{\bP} 

\def\funct lp{L} 
\def\funct lpbar{\bar L} 

\definecolor{darkgreen}{rgb}{0.0,0.5,0.0}
\definecolor{darkblue}{rgb}{0.0,0.0,0.3}
\definecolor{nicosred}{rgb}{0.65,0.1,0.1}
\definecolor{light-gray}{gray}{0.7}
\allowdisplaybreaks[1]

\RequirePackage{datetime} 

\begin{document}

\title[Mittag-Leffler queues]
{Queuing models with Mittag-Leffler inter-event times}

\author{Jacob Butt}
\address{Jacob Butt, University of Sussex, Brighton, UK}
\email{J.Butt@sussex.ac.uk}
\author{Nicos Georgiou}
\address{Nicos Georgiou, University of Sussex, Brighton, UK.}
\email{N.Georgiou@sussex.ac.uk}
\author{Enrico Scalas}
\address{Enrico Scalas, University of Sussex, Brighton, UK and University of Rome `La Sapienza', Rome, Italy}
\email{E.Scalas@sussex.ac.uk}
\email{enrico.scalas@uniroma1.it}

\keywords{GI/ GI/1 queue, Mittag-Leffler queues, queue length, recurrence, transience, fractional derivatives, time-changed queue, semi-Markov process, mixing times, scaling limits for the queue length}
\subjclass[2000]{Primary: 60K25, 60F17, Secondary: 60K15} 
\date{\today}
\begin{abstract}
We study three non-equivalent queueing models in continuous time that each generalise the classical M/M/1 queue in a different way. Inter-event times in all models are Mittag-Leffler distributed, which is a heavy tail distribution with no moments. For each of the models we answer the question of the queue being at zero infinitely often (the `recurrence' regime) or not (the transient regime). Aside from this question, the different analytical properties of each models allow us to answer a number of questions such as existence and description of equilibrium distributions, mixing times, asymptotic behaviour of return probabilities and moments and functional limit theorems.   
\end{abstract}
\maketitle

\section{Introduction}
The theory and application of models for queuing systems became an integral part of many disciplines, playing an important role in areas such as health-care \cite{Fomundam07}, telecommunications \cite{Giambene14}, and operations research \cite{Larson81} among others. The theoretical basis for much of modern queuing theory lies in the extensively studied M/M/1 queue model \cite{Abate88, Bose13, Kendall53}. This model uses exponentially distributed inter-arrival and service times for the customers, rendering it amenable to rigorous mathematical treatment and making it a good starting point for many situations. Being a simple model, however, means it lacks a number of properties such as preservation of memory, which are often desirable when studying more complex systems. There are several examples in which emory effects are necessary; for example in studying hysteresis effects in biological and epidemiological models \cite{Kopfova06}.

One natural extension is to consider more general distributions for the arrival and service times, leading to GI/GI/1 queue models (see \cite{Medhi02} section 7). The analysis of these models becomes more robust when inter-arrival and service times have moments, but it becomes more challenging in the absence of all moments. Nonetheless, heavy-tailed distributions appear naturally in queuing models, see  for example \cite{Foss, Whitt2}.
One particular question of interest is whether the queue empties infinitely often and the length has an equilibrium distribution when the inter-arrival and service times do not have all moments. This stability of the first-in, first out (or first come -first served) queue is characterised by the value of the load (or the traffic intensity), which is the ratio of the first moments of inter-arrival and service times so a different version of this characterisation is needed when both of these moments do not exist.   

Different methods for handling queuing models with heavy-tailed distributions have began to emerge of late  \cite{Carpinteri97}, using techniques from the field of fractional calculus. These techniques are particularly useful when the model has Mittag-Leffler distributions for waiting times \cite{Gorenflo20, Mainardi14, Mieghem20} as the connection of the Mittag-Leffler function to fractional calculus is well understood \cite{Mainardi20}. Its properties make the models that use the Mittag-Leffler distribution amenable to rigorous mathematical analysis and exact results.
Moreover, It has become increasingly common to introduce fractional operators into standard models in order to introduce memory effects. These models are also developed with applications in mind; for example see \cite{Sun18} and \cite{Tarasov19} which give a good overview of some applications in engineering and finance respectively.

Recently, in \cite{Cahoy15}, the Caputo fractional derivative appeared in the system of differential equations for the M/M/1 queue. In this fractional queue model, the inter-arrival and service times are now characterised using the Mittag-Leffler distribution with parameter $\alpha \in (0,1]$, where setting $\alpha = 1$ recovers the classical model. It is shown that this model can be described analytically as a time-changed version of the standard M/M/1 queue through the use of an inverse stable subordinator. These subordinators are often used in the study of fractional models and introduce a random timescale and memory effects into the model  \cite{ Meerschaert13, Meerschaert11}. In particular, when $\alpha \in (0,1)$, the model now loses the memoryless property associated with the classical M/M/1 queue and becomes a semi-Markov model.
In addition, modifications of the model in \cite{Cahoy15} and similar fractional queue models have been studied in the recent past using the same subordinator methods, including the addition of catastrophes to the model \cite{Ascione18}, as well as a related fractional Erlang queue model \cite{Ascione20} and a fractional birth-death model \cite{Curinao22, Orsingher11} which can be used to model the queue length of a time-changed queue.

In this article we study three different ways to introduce a single server queue with Mittag-Leffler inter-event times. The first one is precisely the time-changed  M/M/1 model in which the evolution of the system depends on events of a background Poisson process. Whether the Poisson event is an arrival or a departure is decided by a coin flip. While the study of this queue falls under the  \cite{Ascione20, Cahoy15, Curinao22, Orsingher11}, we offer an expression and evolution equation for the moment generating of the queue length, as well as asymptotic behaviour for return probabilities, moments and mixing times to equilibrium. The second model is a generalisation of a Gillespie-type queue which has a renewal structure. After every event, competing independent waiting times re-start for arrivals and departures, and the fastest one decides the type of event. Finally, the third model is the fully generalised GI/GI/1 queue with Mittag-Leffler distributed inter-arrival and service times. A crucial question for all models is that of recurrence and transience of the queue length, where the notions are suitably interpreted for these models.

The original motivation for this article came from a need to predict certain financial markets. Over the past few decades numerous studies have been conducted aiming to examine the behaviour of times between trades in financial markets. In a number of cases (see for example \cite{Raberto02, Sabatelli02}) waiting times have been found to match nicely with heavy tailed distributions such as the Mittag-Leffler distribution (i.e. they exhibit no moments at all).  As models for such markets (referring in particular to models such as the Double Auction models in \cite{Cont10, Radivojevic14}) can be related to GI/GI/1 queue models, we expect the development of queueing systems with Mittag-Leffler distributed waiting times to be useful in the near future.

\subsection{Structure of the article:} In Section \ref{sec:preliminaries} we describe three basic equivalent descriptions of the Markovian M/M/1. The equivalence of these three models hinges on the Markov property. Each of these will lead to a distributionally different fractional queuing model. 
In Section \ref{sec:results} we record the results for each of the models. The proofs of these results is the rest of the article.  In Section \ref{sec:model1} we revisit the fractional model in \cite{Cahoy15} where we prove the fractional evolution of the moment generating function as well as mixing times to equilibrium in the recurrent regime. The proofs that are closely mimicking existing proofs in the literature are postponed to Appendices \ref{sec:A}, \ref{sec:B} and \ref{sec:C}, however we do include them for completion and easy access. Section \ref{sec:model2} is dedicated to the renewal type queue. Because of the renewal structure this one can be studied with classical tools. 
Finally in Section \ref{sec:model3} we study the fully generalised GI/GI/1 queue with Mittag-Leffler distributed arrival and service times. We establish and prove a functional limit theorem for this generalised model, and classify transient and recurrent regimes in this context.

\subsection{Acknowledgments} Jacob Butt, Nicos Georgiou and Enrico Scalas would like to thank the Isaac Newton Institute for Mathematical Sciences, Cambridge, for support and hospitality during the programme Fractional Differential Equations (FDE2) where work on this paper was undertaken. This work was supported by EPSRC grant no EP/R014604/1. Nicos Georgiou and Enrico Scalas was also partially funded by the Dr.\ Perry James (Jim) Browne Research Center at the Department of Mathematics, University of Sussex.

\section{Preliminaries}
\label{sec:preliminaries}

	\subsection{Three equivalent M/M/1 models} 
	
	In this article we consider three generalisations of the classical M/M/1 queue. 
	Each of the models we are considering are natural extensions of different equivalent ways to define the continuous time 
	M/M/1 queue. 
	
	All three classical models can be viewed as a time-change of the discrete queue, which we present first, in full generality.  
	
	Let $\{ Q_n \}_{n \ge 0}$ denote a Markov chain on the non-negative integers, 
	which represents the length of the queue 
	at discrete time $t = n$. Initially, at $n=0$ the queue has $Q_0 = i_0 \in \Z_+$ and after
	 that it evolves according to the homogeneous transition probabilities 
	 \[
	 p_{i,j} = P\{ Q_{n+1} = j | Q_n = i\}, 
	 \] 
	given by
	\begin{align} \label{Transition probilities}
   		 p_{k,k} &= \beta, \\
   		 p_{k,k+1} &= \begin{cases}
     					   1 - \beta,  &\text{ if } k = 0 \\
     					   (1-\beta) p,  &\text{ otherwise,}
  					  \end{cases} \\
  		 p_{k, k-1} &= \begin{cases}
      					  0, &\text{ if } k = 0 \\
        					(1-\beta) (1-p) &\text{ otherwise.}
     				 \end{cases}
	\end{align}
	All other transition probabilities equal zero. 
	Parameter $\beta$ can equal 0. This just means that there is either an arrival or a departure in any time interval 
	in which the queue length is not empty. 
	It can also equal 1 but in that case the queue does not evolve. 
 
	Ergodicity properties (or lack of) for the Markov chain are well-known (see, for instance, \cite{Norris98}),
	as the transition probabilities $\{p_{i,j}\}_{i,j \in \Z_+^2}$ represent a discrete \emph{birth and death chain.} 
	In particular we have that 
		\[
		Q_n \text{ is positive recurrent } \Longleftrightarrow p < 1/2, \quad 
		Q_n \text{ is null recurrent } \Longleftrightarrow p = 1/2.
		\]
	In these cases the queue empties infinitely often (as $Q_n= 0$ for infinitely many $n$) and when $Q_n$ 
	is positive recurrent, the queue has an equilibrium and unique invariant distribution that is computable and given by
		\be \label{eq:equilibriumform}
		\lim_{n\to \infty} \P\{ Q_n = j \}= \pi_j = \frac{1-2p}{1-p}\Big(\frac{p}{1-p}\Big)^{j}.
		\ee
	
	The discrete Markovian queue can be made into a continuous time Markovian queue in several 
	equivalent ways, three of which we present here. In the sequel, when we create a time-fractional queue, 
	we will have that each of these ways lead to a different (non-Markovian) queue.  
	
	\textbf{Model 1: Time-changed discrete chain.}  A way to define the continuous time $M/M/1$ queue is by 
	performing a global time change on the discrete model. To this end, let $N_1(t)$ denote a Poisson process 
	with rate $1$, and fix the parameters of the discrete model to $\beta = 0$ and $p>0$. Then, the length $L_t^{(1)}$ 
	is defined as  
	\be \label{eqn: model 1 equality}
	L_t^{(1)} = Q_{N_{1}(t)}.
	\ee    
	The equality above is an a.s.\ equality in the following way: At the time 
	of the $n$-th Poisson event, in order to decide if it corresponds to an arrival or a departure, 
	look at the type of event the discrete queue  (whether it is arrival or departure), and that is the type of event 
	the continuous queue will have.   
	
	There is nothing special about the rate of the Poisson process being 1, but it motivates the construction of Model 3 below.
	\\
	
	\textbf{Model 2: The renewal approach.} Consider two independent i.i.d.\ sequences of exponential random variables 
	with rate parameters $\lambda$ and $\mu$ respectively, $\{ X_{i}^{(\lambda)}\}_{i \ge 1}$ and  $\{ X_{i}^{(\mu)}\}_{i \ge 1}$. 
	Define 
	\[T_i = X_i^{(\lambda)}\wedge X_i^{(\mu)} \sim \text{Exp}(\lambda+ \mu).\]
	Then for any $t>0$, define  
	\[
	N_{\lambda+\mu}(t) = \max\Big\{ k: \sum_{i=1}^{k} T_i \le t <  \sum_{i=1}^{k+1} T_i  \Big\}.
	\]
	As above, $N(t)$ is a Poisson Process with rate $\lambda+\mu$, but constructed slightly differently 
	from Model 1. Moreover, when the Poisson process ticks for the $n$-th time, we interpret the events as  
	\[
	\text{ arrival of a customer at the $n$-th event } \Longleftrightarrow X_i^{(\lambda)}\wedge X_i^{(\mu)} = X_i^{(\lambda)},
	\]
	otherwise we have a completion of service. Then the queue length $L^{(2)}_t$ at time $t$ is given by 
	\be
	L^{(2)}_t = Q_{N_{\lambda+\mu}(t)},
	\ee
	where the parameters for the discrete (embedded) chain $Q_n$ are  
	$p = \lambda(\lambda+\mu)^{-1} = \P\{ X_i^{(\lambda)}\wedge X_i^{(\mu)} = X_i^{(\lambda)}\}$ and $\beta = 0$.
	\\
	
	\textbf{Model 3: Independent arrivals and departures.}
	For a distributional equality in \eqref{eqn: model 1 equality} of Model 1, 
	without looking at the evolution of $\{Q_n\}_{n \ge 1}$, go through the events of $N_1(t)$ 
	and flip an independent coin for each one (and independent of $N_1(t)$), marking the Poisson events as success or failure. 
	Let $p$ denote the probability of success and $1-p$ the probability of failure.  
	This creates an arrival and a departure process $N_p(t)$ and $N_{1-p}(t)$ respectively. 
	These are \emph{thinned out} Poisson processes and are independent of each other. 
	
	 Every time $N_p(t)$ ticks, there is an arrival in the system while every time $N_{1-p}(t)$ ticks, 
	 there is a completion of a service time and, if there is a customer in the system, they exit reducing the queue length by 1.  
	 The queue length in this case is 
	\be
	L^{(3)}_t = Q_{N_{p}(t)+N_{1-p}(t)} \stackrel{\mathcal D}{=} Q_{N_{1}(t)} = L^{(1)}_t .
	\ee

	This can be generalised to having two independent processes $N_{\lambda}(t)$ and $N_{\mu}(t)$ representing 
	the arrival process and the departure process respectively. Then  using $p = \lambda(\lambda + \mu)^{-1}$ and 
	$\beta = 0$ for the discrete queue $\{Q_n\}_{n \ge 1}$ we have 
	\be
	L^{(3)}_t = Q_{N_{\lambda}(t)+N_{\mu}(t)} \stackrel{\mathcal D}{=} Q_{N_{\lambda+\mu}(t)} .
	\ee
	
	Moreover, using the interpretation of the model that we have two independent arrival and departure processes 
	there is an explicit and useful formula calculating the queue length at any given time  
	\be \label{two queue relation}
   		 L_t = \left(N_{\lambda}(t) - N_{\mu}(t)\right) 
		 - \inf_{0 \le s \le t}\left\{N_{\lambda}(s) - N_{\mu}(s) \right\}.  
	\ee
	The proof of this a.s.~equality goes by induction on the number of Poisson points in the interval $[0,t]$. 
	As such, the proof works with any two counting processes to represent arrivals and departures, 
	not just when these are Poisson processes, as long as the number of events in any time interval is almost surely finite.  
	A proof of equation \eqref{two queue relation} can be found in the Appendix.
	
		 			\begin{figure}[h]
  \includegraphics[width=\linewidth]{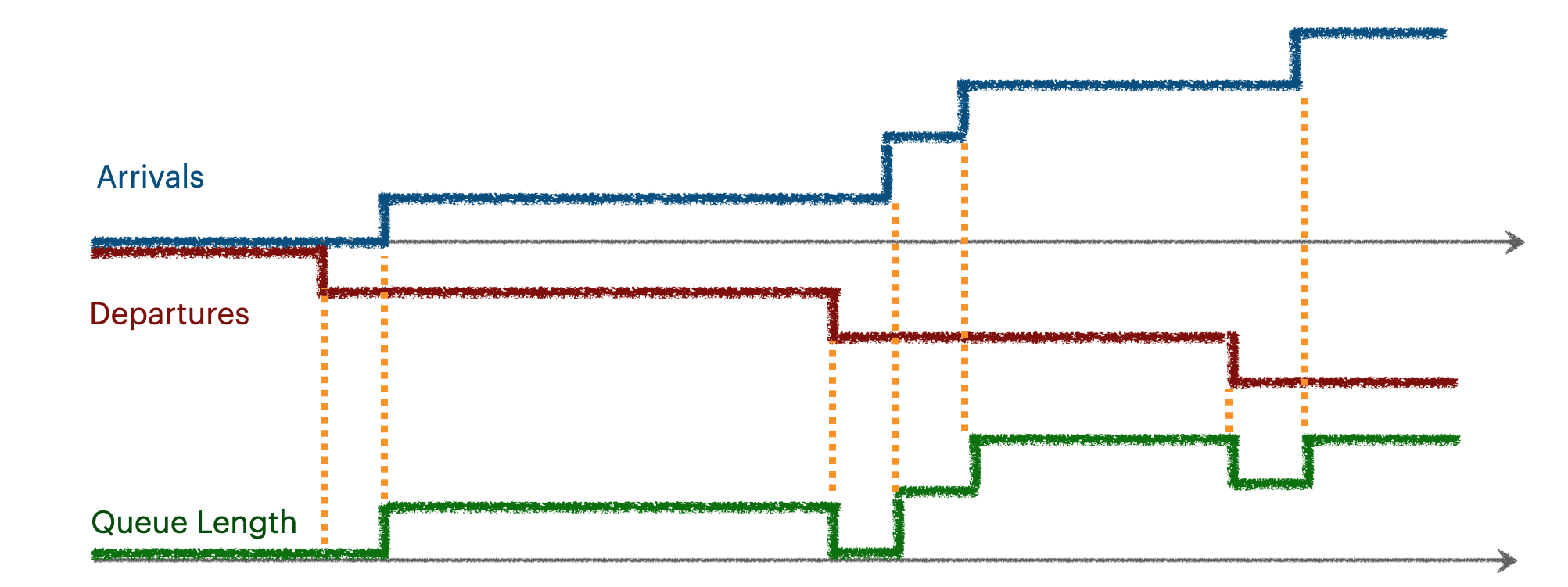}
  \caption{A schematic of the queue length formula as a reflected process of the difference between arrivals and departures} 
  \label{fig:schema}
\end{figure}

\section {Results}
\label{sec:results}

\subsection{Heavy tailed models}

	As stated earlier, we will be extending each of the three equivalent 
	Markovian queuing systems into queuing systems that are generated from 
	Mittag-Leffler inter-event times. The three generalisations however will not be equivalent, 
	and this demonstrates the need to find the correct fractional model in applications that 
	demonstrate heavy tailed inter-event times. We first record some properties 
	of the Mittag-Leffler function and distribution that are used throughout this article, 
	followed by our results in each of the three models. An extensive review of the Mittag-Leffler 
	function can be found in \cite{Mieghem20} (see also the references therein).
	
	The density of the Mittag-Leffler distribution of power index $\alpha \in (0,1]$ is
	\[
	f_\alpha(x) = x^{\alpha-1}E_{\alpha, \alpha}(-x^\alpha) 
	= x^{\alpha -1}\sum_{k=0}^\infty\frac{(-1)^kx^{\alpha k}}{\Gamma(\alpha+\alpha k)},
	\]
	with c.d.f.\ given by 
	\[
	F_\alpha(x) = \P\{X_{\alpha} \le x\} 
	= 1- E_{\alpha,1}(-x^\alpha) = 1- \sum_{\ell=0}^\infty\frac{(-1)^\ell x^{\alpha \ell}}{\Gamma(1+\alpha \ell)}.
	\]
	Note that the power series above converge absolutely, since the generalised Mittag-Leffler function 
	with parameters $\alpha, \beta$ 
	\[
	E_{\alpha, \beta}(z) = \sum_{\ell=0}^\infty\frac{z^{ \ell}}{\Gamma(\beta+\alpha \ell)}
	\]
	also does. 
	
	When $\alpha = 1$ the density and c.d.f.\ above coincide with those of the exponential distribution of 
	parameter 1. We can scale the Mittag-Leffler distribution the same way as the exponential 
	distribution can be scaled, namely for $\lambda>0$
	\be \label{eq:outscale}
	X_{\alpha, \lambda} = \frac{1}{\lambda} X_{\alpha}.
	\ee
	Its density and c.d.f.\ are given by 
	\be\label{eq:ml-time-scale}
	f_{\alpha, \lambda}(x) = \lambda^\alpha x^{\alpha-1}E_{\alpha, \alpha}(-\lambda^\alpha x^\alpha), 
	\quad F_\alpha(x) =  1- E_{\alpha,1}(-\lambda^\alpha x^\alpha).
	\ee
	
	\begin{remark}
	The usual way we see the scaling $\lambda$ in the literature is that the c.d.f.~is given as 
	\[F_\alpha(x) =  1- E_{\alpha,1}(-\lambda x^\alpha).\] 
	This would correspond to  
	\[
	X_{\alpha, \lambda} = \frac{1}{\lambda^{1/\alpha}} X_{\alpha},
	\]
	which is notationally slightly more cumbersome for our purposes. So in this article we are using \eqref{eq:outscale} which will make the index $\alpha$ appear as in \eqref{eq:ml-time-scale}. \qed
	\end{remark}
	
	Finally, we are using the following asymptotic property of $E_{\alpha,1}$, 
	\be\label{eq:asym1}
	1- F_{\alpha}(x) = E_{\alpha,1}(-x^\alpha) \sim \frac{\sin(\alpha \pi) \Gamma(\alpha)}{\pi}x^{-\alpha}, \quad x \to \infty,
	\ee
	without any particular mention.
		\subsection{The (single) time fractional queue}
	
	For this model, we use the semi-Markov approach taken in \cite{Georgiou15} for a global time change 
	on the discrete queue process; the interested reader can see \cite{Gikhman04} for a more general overview of such methods. It is equivalent to the counting process being a Fractional Poisson Process 
	$N^{(\alpha)}_{\lambda}(t)$, 
	where inter-event times 
	$T_i \sim X_{\alpha, \lambda}$ are Mittag-Leffler distributed with power index $\alpha$ and scale $\lambda$, 
	with density given  by \eqref{eq:ml-time-scale}.
	
	Let the queue length for this model be defined by  
	\be\label{eq:q-length-1}
	L^{(1)}_{\alpha, \lambda}(t)= Q_{N^{(\alpha)}_{\lambda}(t)}.
	\ee 
	For the discrete model in this case we assume that $\beta = 0$ in \eqref{Transition probilities} (just for simplicity) and that $L^{(1)}_{\alpha, \lambda}(0)= Q_{N^{(\alpha)}_{\lambda}(0)}= Q_0 = 0$. 
	 As such, we have the forward Kolmogorov equations for the queue length
	 \be\label{eq:kolm}
    		p_{0,i}(t) = \P\{ L^{(1)}_{\alpha, \lambda}(t) = i\} =
	 		\bar{F}_{{\alpha, \lambda}}(t) \delta_{0}(i) 
	 			+ \sum_{k \in \{(i - 1)\vee 0,  i + 1\}}q_{k, i}\int^t_0 p_{0,k}(u)f_{\alpha, \lambda}(t-u) du.
	\ee
	Above we shorthanded $\bar{F}_{{\alpha, \lambda}}(t)  = 1 - F_{{\alpha, \lambda}}(t)$ 
	for the survival function (c.c.d.f.) of $X_{\alpha, \lambda}$. Equation \eqref{eq:kolm} holds because 
	we are assuming that we are starting from a renewal point at $t=0$. The range of index $k$ comes from the the fact that 
	$q_{k,i} =0$ if $k$ is not in that range.
	Using Laplace transforms, one can show (see Appendix \ref{sec:C} for a direct derivation) that the temporal evolution of $p_{i}(t)$ satisfies the 
	time-fractional differential equation 
	\be \label{fractional kolmogorov eqn}
   	 \frac{d^{\alpha}p_{i}(t)}{dt^{\alpha}} = - \lambda^{\alpha} p_{i}(t) +  \lambda^{\alpha} p p_{i-1}(t) + \lambda^{\alpha} (1-p) p_{i+1}(t), \quad i \ge 1.
	\ee
	Similarly, we can find the boundary conditions 
		\begin{align} \label{fractional kolmogorov boundary eqn}
  			  \frac{d^{\alpha}p_{0}(t)}{dt^{\alpha}} = - \lambda^{\alpha} p p_{0}(t)
			   + \lambda^{\alpha}  (1-p) p_{1}(t).
		\end{align}
	Operator $\displaystyle \frac{d^{\alpha}p_i(t)}{dt^{\alpha}} $ is the Caputo derivative \cite{Daftardar13} of $p_i(t)$. 
	For any function $f$ it is defined by 
	\[
	\frac{d^{\alpha}}{dt^{\alpha}} f(t) = \frac{1}{\Gamma(1 - \alpha)}\int_{0}^t (t - u)^{-\alpha} \frac{d}{du}f(u) \,du.
	\]
	
	As mentioned earlier, this queue is the same as the fractional queue studied by \cite{Cahoy15}. We would like to extend some of their results in this article to recover a full expression for the temporal evolution of the moments of the queue length, as well as provide a proof of convergence to invariant distribution and some asymptotic results.
	
	We offer the formula for the Laplace transform of the moment generating function 
	\[ 
	M_{\alpha, \lambda}(z,t) = \E(e^{zL^{(1)}_{\alpha, \lambda}(t)}) 
	\] 
	where 
	$L^{(1)}_{\alpha, \lambda}(t)= Q_{N^{(\alpha)}_{\lambda}(t)}$ when we assume that the queue is empty at time 
	$t = 0$ and zero is a renewal point for the Mittag-Leffler inter-event times. Define
	 \be\label{r1r2}
	r_{1}(s) 
	=  \frac{1+s^{\alpha} + \sqrt{(1 + s^{\alpha})^2 -4p(1-p)}}{2p}, \quad r_{2}(s) 
	=  \frac{1+s^{\alpha} - \sqrt{(1 + s^{\alpha})^2 -4p(1-p)}}{2p}.
	\ee
	
	\begin{theorem}\label{thm:Q1mgf} Let $p_{0,0} = \P_0\{ L^{(1)}_{\alpha, \lambda}(t) =0 \}$. 
	Then the moment generating function satisfies the fractional differential equation
	\[
	\frac{d^{\alpha}}{dt^{\alpha}}M_{\alpha, 1}(z,t) 
	=(1-p)(1 - e^{-z}) p_{0,0}(t) +  \big(pe^z  - 1  + (1-p)e^{-z}\big)M_{\alpha, 1}(z,t)
	\]
	and has the Laplace transform 
	  \be\label{mom}
	  \widetilde{M}_{\alpha, 1}(z,s) = \mathcal L\{ M_{\alpha, 1}(z,t)\}(s)
	  =\frac{-s^{\alpha-1}}{p(e^z - r_1(s))(1-r_2(s))}.
	  \ee
	\end{theorem}
	
	Particular uses of this are asymptotics for moments, the variance of the queue length and the probability 
	of seeing an empty queue as $t \to \infty$ via Tauberian theorems. 
	These can be found in Examples \ref{ex:emptyQ1}, \ref{ex:emptyQ2} is Section \ref{sec:as}. 
	
	\begin{remark} In general, the queuing representation \eqref{two queue relation} also applies in this setting. 
	Time-changing the two independent Poisson processes with a common inverse stable subordinator 
	will lead to the same queuing system in distribution. The difference of two Fractional Poisson Processes 
	with this common time change can be found in the literature under the name of a fractional Skellam process of type 2,
	see \cite{booboo}. 
	In this respect, the queue is a fractional Skellam process, partially reflected at 0. \qed
	 \end{remark}
	
	Since this queue is obtained by a direct time change of the Markovian one, it is straight forward 
	to argue that  when $p < 1/2$, as in equation \eqref{eq:equilibriumform}, 
	\[
	\lim_{t \to \infty} p_i(t) = \frac{1-2p}{1-p}\Big(\frac{p}{1-p}\Big)^i.
	\]
	Since the state space for the queue length is infinite and time is delayed due to the time change, 
	it is natural to search for quantitative estimates 
	for the time it takes for the $p_i(t)$ to be near its equilibrium value. 
	For this we will be using the total variation distance.
	The total variation
	norm between two measures, $\mu$ and $\nu$, on a common 
	probability space $\Omega$ with $\sigma-$algebra $\mathcal{F}$ is defined as
		\begin{equation}
    			\|\mu - \nu\|_{TV} = \sup_{A \in \mathcal{F}}|\mu(A) - \nu(A)|.
		\end{equation} 
		
	The support space for the case of the queue length is $\bN_0$ 
	which is a countable space,
	so overall the total variation distance can be written more simply as
	\begin{equation}
  	  \|\mu - \nu\|_{TV} = \frac{1}{2}\sum_{x \in \N_0} |\mu(x) - \nu(x)|.
	\end{equation}

	Irreducible discrete Markov chains that are positive recurrent converge to 
	a unique equilibrium and this is the case with $Q_n$ in \eqref{eq:equilibriumform}. 
	The \emph{$\e$-mixing time} $T^{\e}_{\rm mix}$
	for any chain $X_t$ with initial distribution $\mu_0$ and with a unique equilibrium distribution $\pi$ is defined by 
	\be\label{eq:mixdef}
		T^{\e}_{\rm mix} = \inf\{ t > 0: \| \P_{\mu_0}\{ X_t \in \cdot \} - \pi \|_{TV} \le \e \}
	\ee

	Then we have the following theorem. 
	
	\begin{theorem} \label{thm:Q1mixtime}
		Let $L^{(1)}_{\alpha, \lambda}(t)$ be given by \eqref{eq:q-length-1} and let $\pi$ 
		be the equilibrium distribution given by \eqref{eq:equilibriumform}. 
		Assume that the initial queue length distribution $\mu_0$ is finitely supported. 
		Then there exist positive constants 
		$\e_0, U_{\alpha, \lambda, \mu_0}$   
		such that 
		\[
		T^{\e}_{\rm mix}= \inf\{ t > 0: \| \P_{\mu_0}\{ L^{(1)}_{\alpha, \lambda}(t)  \in \cdot \} - \pi \|_{TV} \le \e \} 
		\le U_{\alpha, \lambda, \mu_0} \e^{1/\alpha}, \quad \text{ for all } \e < \e_0.
		\]
	\end{theorem}
	\medskip
	
	\subsection{The renewal queue} For this section we define 
		\be \label{eq:pal}
		p_{\alpha, \beta}^{ \lambda, \mu} = \P\{ X_{\alpha, \lambda} < X_{\beta, \mu} \}.
		\ee
		The random variables $X_{\alpha, \lambda}, X_{\beta, \mu} $ are independent, 
		Mittag-Leffler distributed.
		Consider two independent i.i.d. sequences $\{X^{(i)}_{\alpha, \lambda}\}_{i\in \N}$, 
		$\{X^{(i)}_{\beta, \mu} \}_{i \in \N}$ and define the i.i.d. sequence
			\[ Y_i = X_{\alpha, \lambda} \wedge X_{\beta, \mu}. \]
		The counting process for this model is now a renewal process $R_t$ with inter-event times $Y_i$.	
		A renewal event will be classified as `arrival' with probability  
		$p_{\alpha, \beta}^{ \lambda, \mu}$ and `departure' otherwise. 
		To describe $R_t$ in words, at time $t = 0$ we have two competing Mittag-Leffler distributions 
		and see which one rings first. If $X_{\alpha, \lambda}$ rings first 
		(i.e.  $X_{\alpha, \lambda} <  X_{\beta, \mu})$ then we interpret the event at time $t =Y_1$ 
		as an arrival, otherwise it is a departure attempt. At time $t = Y_1$ the waiting time restarts. 
		
		An interesting aspect for this model is the fact that the competing Mittag-Leffler variables that define
		 $Y_i$ they can allow the inter-event times to have moments. 
		 To see this consider 
		 \[
		 \P\{ Y > t \} = \P\{ X_{\alpha, \lambda} \wedge X_{\beta, \mu} > t \}= \P\{ X_{\alpha, \lambda} > t \}  \P\{ X_{\beta, \mu} > t \}  \sim \frac{C_{\alpha, \beta, \lambda, \mu}}{t^{\alpha+\beta}}, \quad t \to \infty.
		 \]
		Therefore, when $\alpha + \beta > 1$ the inter-event times have moments and the time-changed 
		queue can be treated as with classical renewal queues with a first moment (but not second).
		
		As in Model (1), the queue length is the time-changed queue length 
			\[ L_t^{(2)} = Q_{R_t} \]
		with an arrival event appearing with probability $p_{\alpha, \beta}^{ \lambda, \mu}$.
		
		\begin{proposition} \label{prop:renewal}
			For any choice of parameters $\alpha, \beta \in (0,1)$,
			 there exists a critical  $\rho_{\alpha, \beta}^* \in (0,\infty)$ 
			 such that  the embedded queue length is transient if and only if 
			 $\lambda/\mu > \rho_{\alpha, \beta}^*$. 

			Moreover, if $\lambda/\mu < \rho_{\alpha, \beta}^*$ then the queue length 
			reaches an equilibrium with distribution \eqref{eq:equilibriumform} and 
			parameter $p_{\alpha, \beta}^{\lambda, \mu}$. 

			In the case where $\lambda/\mu = \rho_{\alpha, \beta}^*$, 
			the embedded Markov chain is null recurrent. 
		\end{proposition}
	
	\subsection{ The Mittag-Leffler GI/GI/1 queue}
	
	 In this model, we consider two independent arrival and
	 departure Fractional Poisson processes. The arrival process 
	 $N^{(\alpha_1)}_{\lambda}(t)$ has inter-event times which are Mittag-Leffler distributed with tail parameter 
	 $\alpha_1$ and scaling $\lambda >0$. The departure process  $N^{(\alpha_2)}_{\mu}(t)$ has potentially 
	 different tail exponent $\alpha_2$ and different scaling $\mu$. 
	 
	 A renewal point of the arrival process signifies that a customer 
	 joined the queue while a renewal point on the departure process suggests the completion of a service 
	 and if the queue length is strictly positive at that time it is reduced by 1, otherwise it remains at zero.  
	 
	 We denote the queue length at time $t$ by $L^{\alpha_1, \alpha_2}_{\lambda, \mu}(t)$. If the scalings are 
	 $\lambda = \mu=1$ we omit them from the notation and simply write $L^{\alpha_1, \alpha_2}(t)$ for the length. 
	 
	 In Figure \ref{fig:limits0} we see how the different tail indices can affect the queue length. The three cases show the qualitative differences for the queue length, and all these behaviours are explored in the theorems below in various ways.
	 
	 	 			\begin{figure}[h]
  \includegraphics[width=\linewidth]{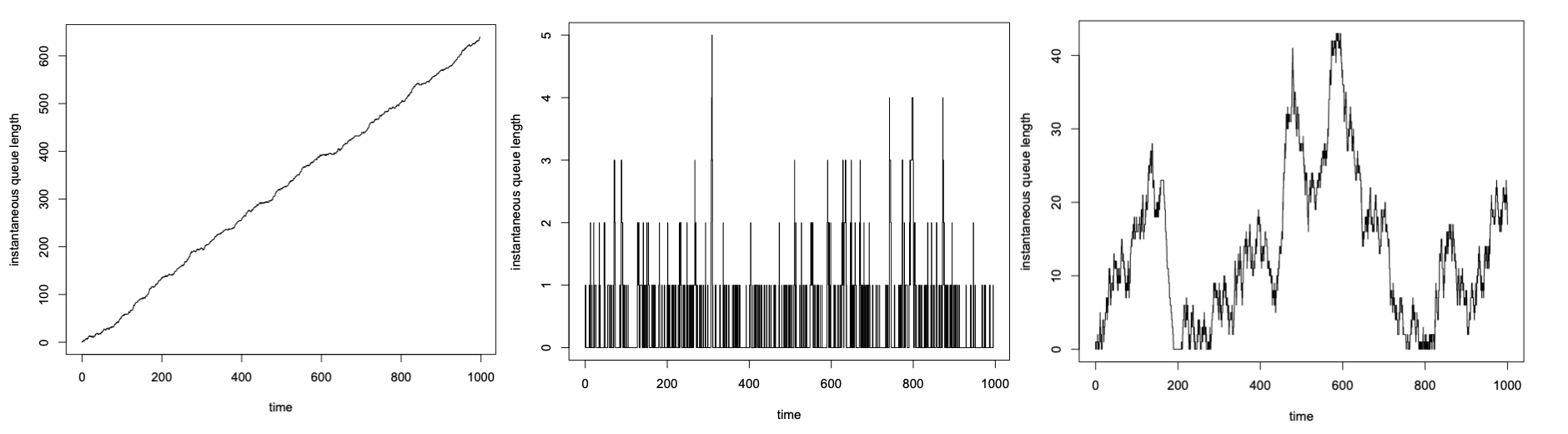}
  \caption{Three simulations for the Mittag-Leffler $GI/GI/1$ queue  for large times.  From left to right are the cases $\alpha_1 > \alpha_2$,  $\alpha_1 < \alpha_2$,  $\alpha_1 = \alpha_2$ respectively.} 
  \label{fig:limits0}
\end{figure}
	 
%

For $\alpha \in (0,1)$ define the $\alpha-$stable subordinator $\left(L_{\alpha}(t)\right)_{t \geq 0}$ to be a positive valued L\'evy process with Laplace transform 
\begin{equation} \label{subordinator defn}
   \bE\left[e^{-\omega L_{\alpha}(t)}\right] = e^{-t \omega^{\alpha}}.
\end{equation}
From this, we can naturally define the inverse $\alpha-$stable subordinator $\left(Y_{\alpha}(t)\right)_{t \geq 0}$ to be given by
\begin{equation} \label{inverse subordinator defn}
    Y_{\alpha}(t) = \inf\{u \geq 0: L_{\alpha}(u) > t\}.
\end{equation}
This inverse $\alpha-$stable subordinator can be used to define the fractional Poisson process as a time change of a Poisson process $N_{\lambda^\alpha}(t)$, by setting 
\begin{equation}
    N_{\lambda}^{(\alpha)}(t) := N_{\lambda^{\alpha}}(Y_{\alpha}(t)).
\end{equation}
One important property to note is that both $L_{\alpha}$ and $Y_{\alpha}$ are self-similar; i.e,
\begin{equation}
    L_{\alpha}(t) \overset{d}{=} t^{1/\alpha}L_{\alpha}(1), \hspace{15.pt} Y_{\alpha} \overset{d}{=} t^{\alpha}Y_{\alpha}(1).
\end{equation}
This inverse $\alpha-$stable subordinator can be used to define the fractional Poisson process as a time change of a Poisson process $N_{\lambda^\alpha}(t)$, as was proved in \cite{Meerschaert11}, by setting 
\begin{equation} \label{FTPP-FPP relation}
    N_{\lambda}^{(\alpha)}(t) := N_{\lambda^{\alpha}}(Y_{\alpha}(t)).
\end{equation}
One important property to note is that both $L_{\alpha}$ and $Y_{\alpha}$ are self-similar; i.e,
\begin{equation} \label{subordinator self-similarity property}
    L_{\alpha}(t) \overset{d}{=} t^{1/\alpha}L_{\alpha}(1), \hspace{15.pt} Y_{\alpha} \overset{d}{=} t^{\alpha}Y_{\alpha}(1).
\end{equation}
This self-similarity property naturally carries over to the fractional Poisson process, where we have 
\begin{align} \label{FPP self-similarity property}
    N_{\lambda}^{(\alpha)}(t) &\overset{d}{=} N_{\lambda^{\alpha}}(Y_{\alpha}(t)) \overset{d}{=} N_1(\lambda^{\alpha} Y_{\alpha}(t)) \overset{d}{=} N_1(Y_{\alpha}(\lambda t)) \overset{d}{=} N_{1}^{(\alpha)}(\lambda t)). 
\end{align}
\begin{align}
    N_{\lambda}^{(\alpha)}(t) &= N_{\lambda^{\alpha}}(Y_{\alpha}(t)) \overset{d}{=} N_1(\lambda^{\alpha} Y_{\alpha}(t)) \overset{d}{=} N_1(Y_{\alpha}(\lambda t)) = N_{1}^{(\alpha)}(\lambda t)). 
\end{align}

We re-write equation \eqref{two queue relation} with this notation. 
\begin{equation} \label{eq:2qnew}
    \left\{L^{\alpha_1, \alpha_2}_{\lambda, \mu}(t)\right\}_{t \geq 0} = \left\{\left(N^{(\alpha_1)}_{\lambda}(t) - N^{(\alpha_2)}_{\mu}(t)\right) - \inf_{0 \leq s \leq t}\left(N^{(\alpha_1)}_{\lambda}(s) - N^{(\alpha_2)}_{\mu}(s)\right)\right\}_{t \geq 0},  
\end{equation}
with $\alpha_1, \alpha_2 \in (0,1)$ and $\lambda, \mu > 0$.
The first theorem in this section concerns the scaling limit of the queue length. 
	 
\begin{theorem}[Scaling limits]  \label{thm: functional limit}
 Let $Y_{\alpha}(t), \tilde{Y}_{\alpha}(t)$ be two independent copies of the inverse $\alpha-$stable subordinator defined in \eqref{inverse subordinator defn}. Let $\gamma = \max\{\alpha_1, \alpha_2\} \in (0,1)$, then we have the following weak convergence as $t \to \infty$ with respect to the Skorohod $J_1$ topology:
\begin{align}
    &\Big\{\frac{L^{\alpha_1, \alpha_2}_{\lambda, \mu}(t\tau)}{t^{\gamma}}\Big\}_{\tau \geq 0} \overset{(w)}{\longrightarrow} \notag \\
    &\qquad\quad\begin{cases}
    \left\{\lambda^{\alpha_1} Y_{\gamma}(\tau)\right\}_{\tau \geq 0} &\text{ if } \alpha_1 > \alpha_2, \\
 \left\{0\right\}_{\tau \geq 0} &\text{ if } \alpha_1 < \alpha_2, \\
 \Big\{\lambda^{\alpha} Y_{\gamma}(\tau) - \mu^{\alpha} \tilde{Y}_{\gamma}(\tau) - \underset{0 \leq s \leq \tau}{\inf}\left(\lambda^{\alpha} Y_{\gamma}(\tau) - \mu^{\alpha} \tilde{Y}_{\gamma}(\tau)\right)\Big\}_{\tau \geq 0}  &\textrm{ if } \alpha_1 = \alpha_2 := \alpha.
 \end{cases}
 \label{eq:lims}
\end{align}
\end{theorem}

\begin{figure}[h]
	\includegraphics[height=5cm]{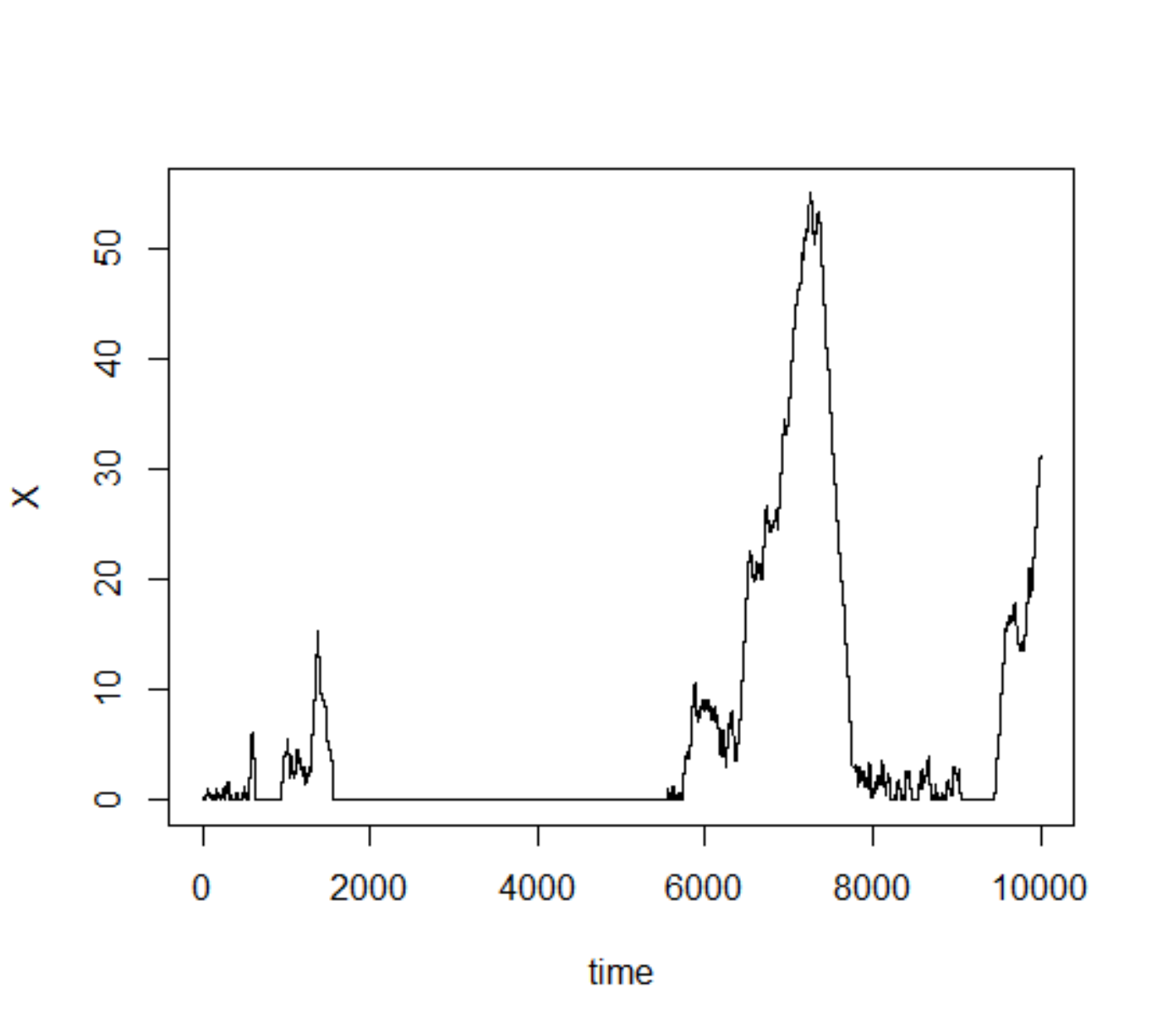}
   \includegraphics[height=5cm]{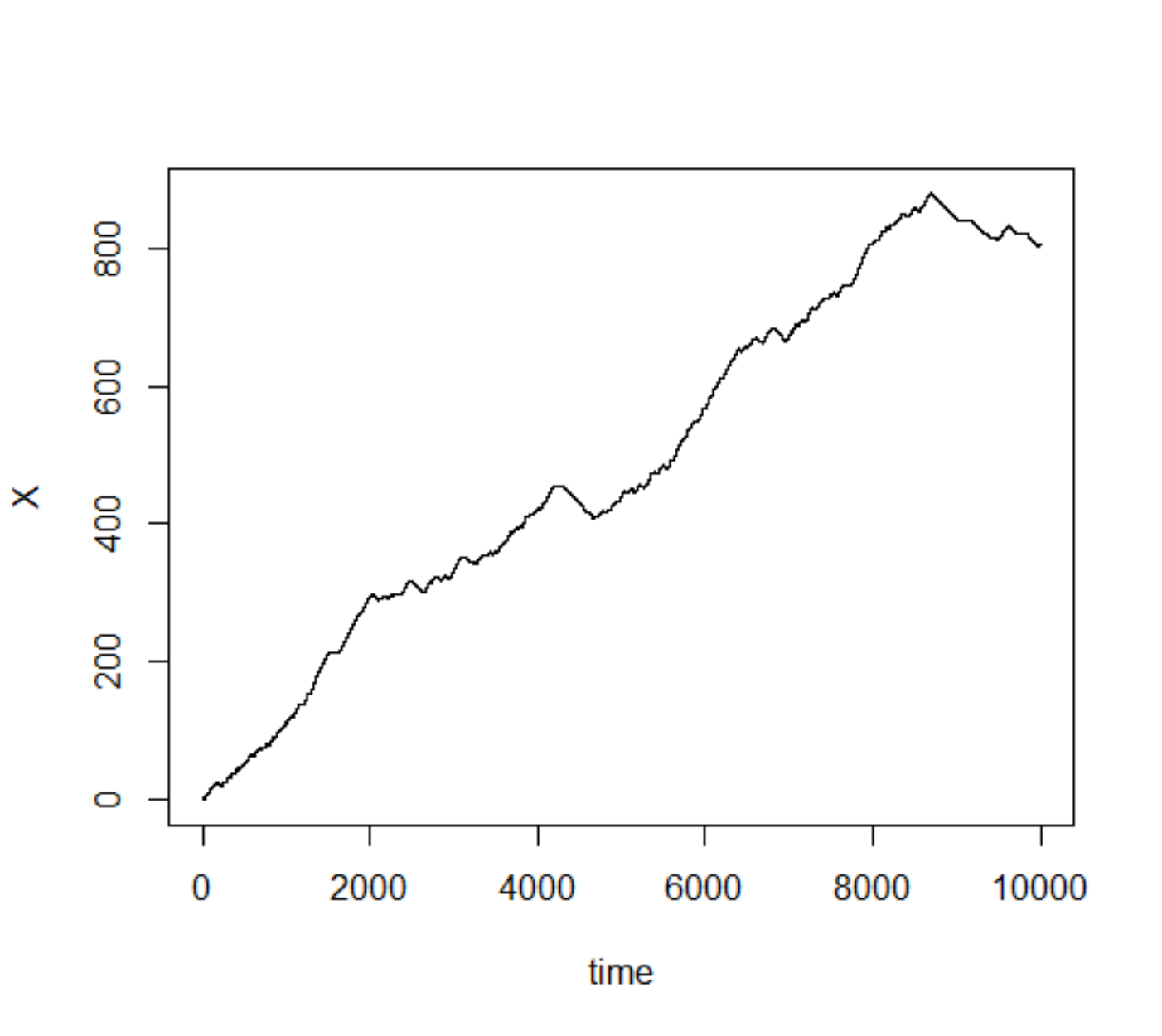}
	\includegraphics[height=5cm]{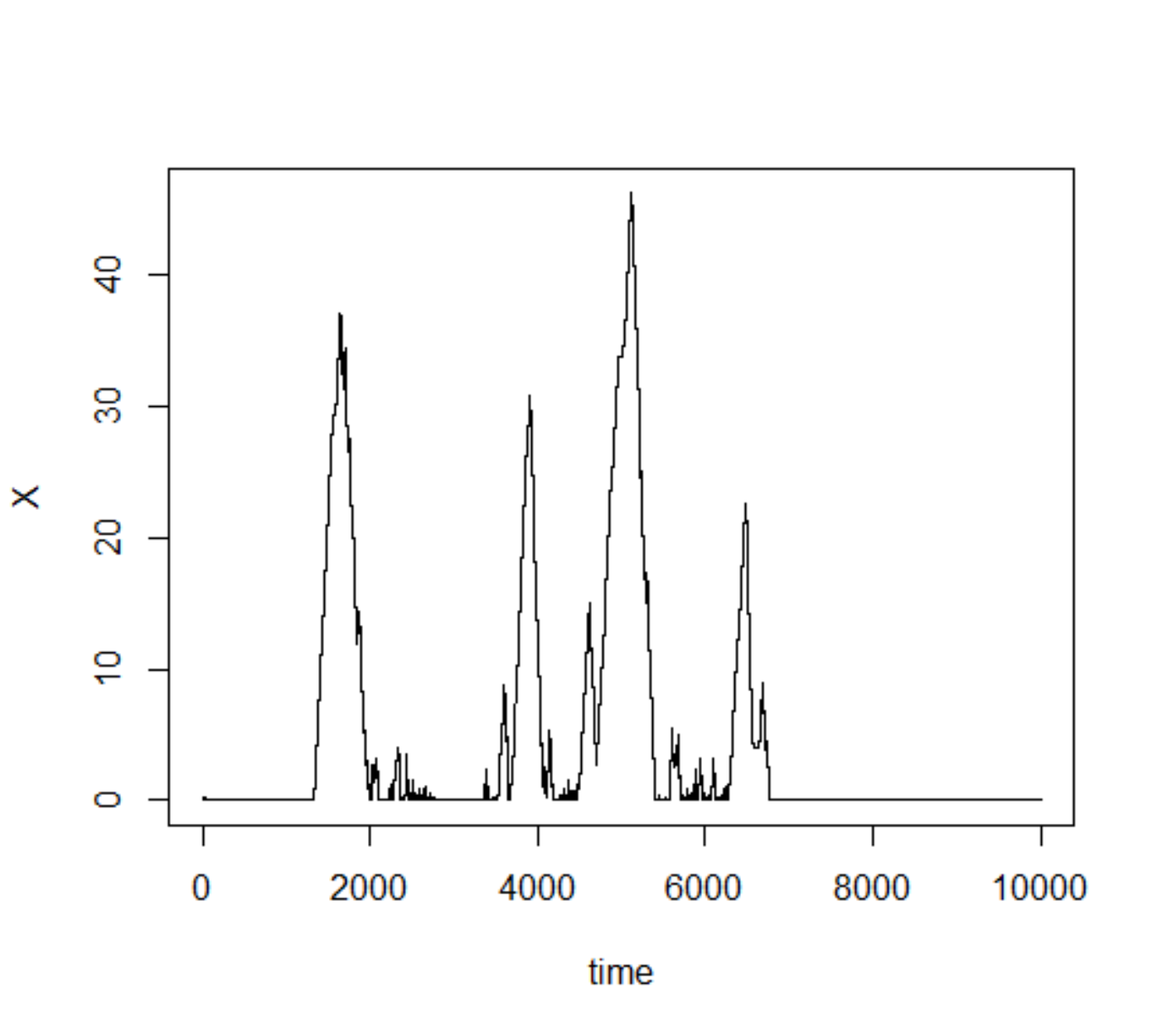}
    \caption{Three simulated realisations for the limiting process  given in the last line of equation \eqref{eq:lims} for the case where $\alpha_1 = \alpha_2 = 0.9$. These figures come directly from the simulations of the inverse alpha-stable subordinator using the CMS formula \cite{Chambers76}. From left to right we have the behaviours for (left) $\lambda = \mu=1$, (middle) $\lambda = 2, \mu =1$ and  (right) $\lambda = 1, \mu =2$.}
    \label{fig:functional limit}
\end{figure}
	 	
		Theorem \ref{thm: functional limit} already suggest the potential behaviour of the queue length in terms of recurrence and transience. Recurrence in this case would mean that the queue length would be zero infinitely often while transience means that the queue length will diverge to infinity as time goes by.  Since the subordinator is non-negative and increasing, we expect that when $\alpha_1 > \alpha_2$ the queue length should divert to infinity. Similarly, when $\alpha_1 < \alpha_2$ and the limit is 0, we expect the queue length to not grow by much (though at this point the theorem only guarantees that it grows less than $t^{\gamma}$). See Figure \ref{fig:limits0} for the corresponding behaviours when the queue length is unscaled. 
		
		Finally, when the tail indices are the same the behaviour should be guided by the scalings $\lambda$ and $\mu$ so the situation is slightly more delicate. This can also be seen in the simulations of Figure \ref{fig:functional limit}. 
		
		The two following theorems explore these various transience and recurrence behaviour.  
		
	\begin{theorem}\label{thm:Rec+Trans} Let $L^{\alpha_1, \alpha_2}_{\lambda, \mu}$ denote the Mittag-Leffler $\rm GI/GI/1$ queue length, as defined in equation \eqref{eq:2qnew}. Then, 
		\begin{enumerate}
			\item If $\alpha_1 < \alpha_2$ then 
			\[
			\P\{ L^{\alpha_1, \alpha_2}_{\lambda, \mu}(t) = 0 \text{ i.o.} \} =1.
			\]
			\item If $\alpha_1 > \alpha_2$ then 
			\[
			\P\{ \varlimsup_{t\to \infty} L^{\alpha_1, \alpha_2}_{\lambda, \mu}(t) = \infty \} =1.
			\]
		\end{enumerate}	
	\end{theorem}

				\begin{figure}[h]
  \includegraphics[width=\linewidth]{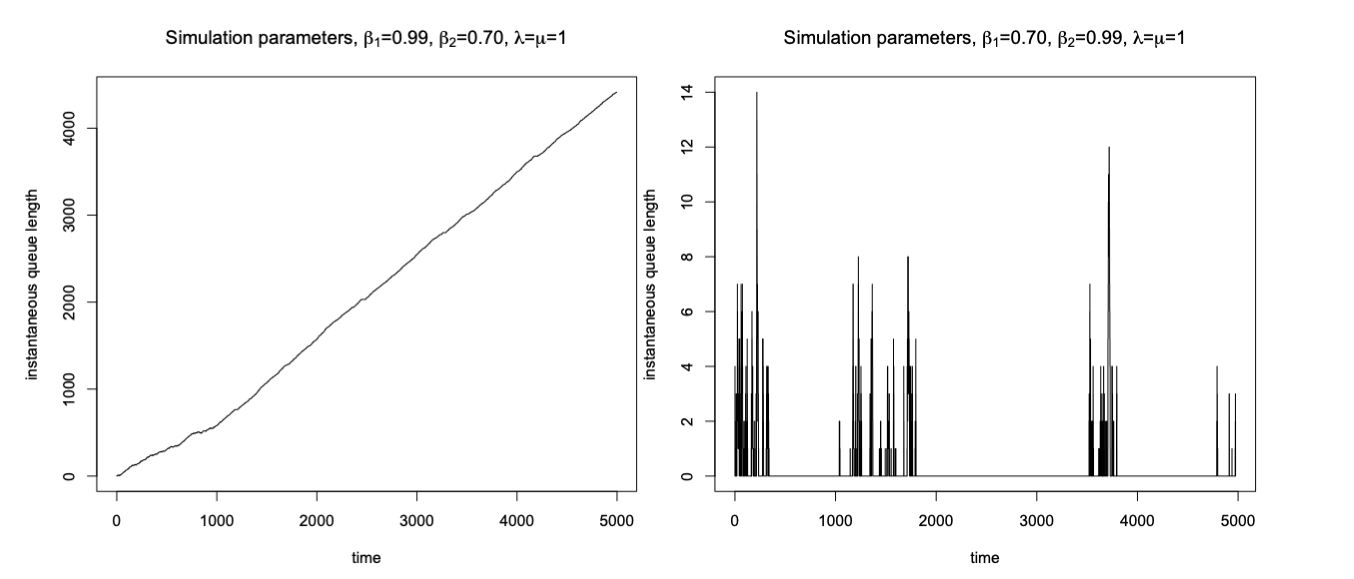}
  \caption{Two simulations for the Mittag-Leffler $\rm{ GI/GI/1}$ queue which highlight their recurrence and transience properties. The arrival and departure counting process have different tail index $\alpha$. These are described in Theorem \ref{thm:Rec+Trans}.}
  \label{fig:powers}
\end{figure}
			
	\begin{theorem}\label{thm:Rec+Trans=Balanced} Let $L^{\alpha_1, \alpha_2}_{\lambda, \mu}$ denote the Mittag-Leffler $\rm GI/GI/1$ queue length, as defined in equation \eqref{eq:2qnew}. Then,
		\begin{enumerate}
			\item If $\alpha_1 = \alpha_2 = \alpha$ and $\mu \ge \lambda$, 
			then 
			\[
			\P\{ L^{\alpha, \alpha}_{\lambda, \mu}(t) = 0 \text{ i.o. } \} =1.
			\]
			\item  If $\alpha_1 = \alpha_2 = \alpha$ and $\mu \le \lambda$, then
			we can find a positive constant $c_{\mu, \lambda}$ so that 
			\[
			\P\Big\{ \varlimsup_{t \to \infty} L^{\alpha, \alpha}_{\lambda, \mu}(t) = \infty \Big\} 
			> c_{\lambda, \mu}.
			\]

		\end{enumerate}
	\end{theorem}
	
	\begin{figure}[h]
  \includegraphics[width=\linewidth]{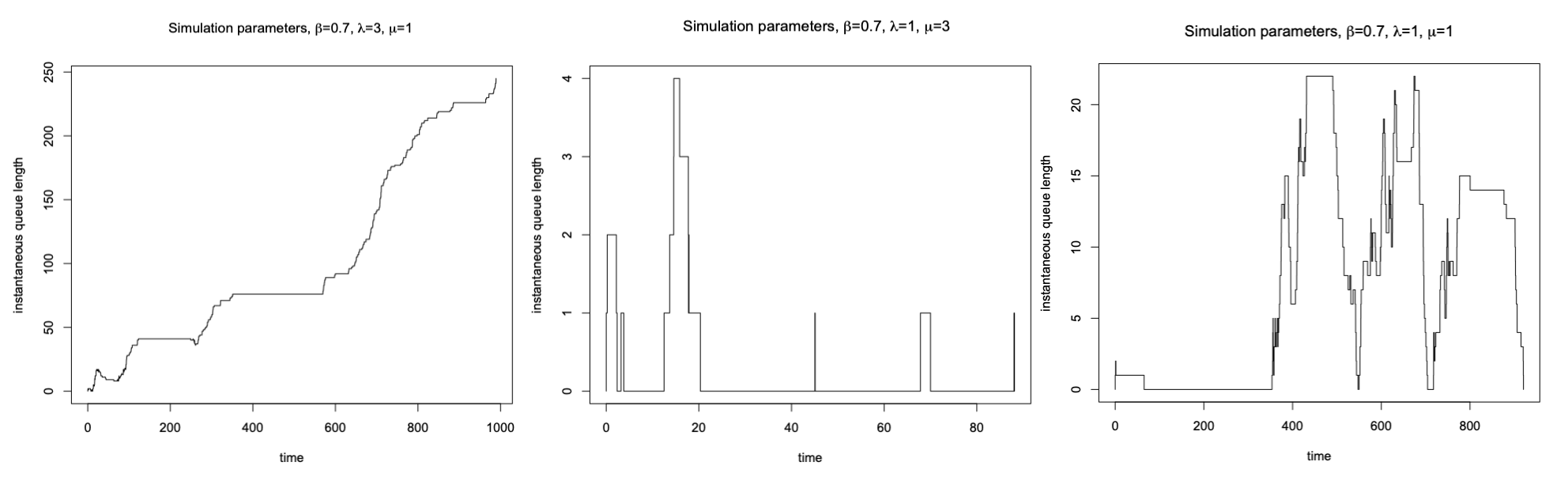}
  \caption{Three simulations for the Mittag-Leffler GI/GI/1 queue. The arrival and departure counting process have the same tail index $\alpha=0.7$ but the arrival rates differ. This is the situation of Theorem \ref{thm:Rec+Trans=Balanced}.}
  \label{fig:queues}
\end{figure}
\begin{remark}
Note that in Figure \ref{fig:queues} it seems that when $\lambda > \mu$ we should be able to substitute $\varlimsup$ with a $\varliminf$ in the statement of Theorem \ref{thm:Rec+Trans=Balanced}, or upgrade $c_{\lambda, \mu}$ with 1. Similarly, when $\lambda = \mu$ we should be able to say $c_{\lambda, \mu} =1$. It seems that both of these arguments require a stronger coupling and comparisons with random walks than the one we actually develop in the sequel. The main difficulty is that these random walks have heavy-tailed continuous increments and are subject to unintuitive behaviours.   
\end{remark}
	
\section{The (single) time fractional queue}
	\label{sec:model1}

	\subsection{Moment generating function}	
	The first thing we would like to establish is that the queue length
	$ L^{(1)}_{\alpha, \lambda}(t)$ with
	mass function 
	$\{p_{0,i}(t)\}_{i \ge 0}$ has exponential moments in a neighbourhood of 0 for any fixed $t$. 
	To this end, let $s > 0$ and estimate
	\begin{align*} 
	\E(e^{s L^{(1)}_{\alpha, \lambda}(t)}) 
	&= \E\Big(e^{s Q_{N_{\lambda}^{(\alpha)}(t)}}\Big) 
	\le \E(e^{s N^{(\alpha)}_\lambda(t)}) 
	= E_{\alpha, 1}(\lambda^{\alpha} t^{\alpha} (e^{s}-1)) < \infty, \quad \text{since } Q_n \le n.
	\end{align*}
	
	This is enough to guarantee that the moment generating function is well-defined for $s \in \R$.  
	
	The interested reader can see a direct proof of Theorem \ref{thm:Q1mgf} in Appendix \ref{sec:B}. The steps of the proof follow the methodology of \cite{Cahoy15} which is written for the probability generating function for the general fractional queue. These steps also imitate classical methodology for the M/M/1 queue \cite{Bailey54}. The Laplace transform of the p.g.f.~for the queue with starting length $i \in \bN$ is explicitly computed in \cite{Cahoy15} and is given by the formula
	
	\begin{equation} \label{eqn: single queue pgf LT}
	     \widetilde{G}_{\alpha, 1}(z,s) = s^{\alpha - 1} \frac{z^{i + 1} - (1-z)r_2(s)^{i+1}(1-r_2(s))^{-1}}{-p(z-r_2(s))(z-r_2(s))}.
	\end{equation}
	Here, $r_1(s), r_2(s)$ are those from equation \eqref{r1r2}.
	This is enough to give us the Laplace transform of all probabilities 
	 \be \label{eq:pon}
	  \widetilde{p}_{0,n}(s) = \frac{s^{\alpha-1}}{p(1-r_2(s))r_1(s)} \frac{1}{r_1^n(s)} = \frac{  \widetilde{p}_{0,0}(s) }{r_1^n(s)}, \quad n \ge 0,
	 \ee 
	 in terms of $ \widetilde{p}_{0,0}(s) $ which is given by (see also  \eqref{eq:lappoo})
	 \be  \label{eq:pon2}
	  \widetilde{p}_{0,0}(s) = \frac{s^{\alpha-1}}{p(1-r_2(s))r_1(s)} 
	  = \frac{r_2(s)s^{\alpha-1}}{(1-p)(1-r_2(s))} = \frac{1}{s} - \frac{p}{1-p} \frac{r_2(s)}{s}.
	  \ee
	
	The proof of Theorem \ref{thm:Q1mgf} is postponed to the appendix 	as the steps  are similar to those of the p.g.f.~ 
	and we continue the section with applications and examples by considering equation \eqref{mom} as given.
		  
	 \subsection{Applications for asymptotics} 
	 \label{sec:as}
	 At this particular point we can expand $ \widetilde{M}_{\alpha, 1}(z,s) $ in to two different series, 
	 for two different applications. First, 
	 \be \label{eq:mlap1}
	 \widetilde{M}_{\alpha, 1}(z,s) = \frac{s^{\alpha-1}}{p(1-r_2(s))r_1(s)}\sum_{n=0}^{\infty}\frac{1}{r_1^n(s)}e^{nz}.
	 \ee
	This is obtained by taking the Laplace transform of the series for ${M}_{\alpha, 1}(z,t)$ and using \eqref{eq:pon}, \eqref{eq:pon2}.
	
	 Second, we can expand \eqref{eq:mlap1} in powers of $z$. We have 
	 \begin{align}
	  \widetilde{M}_{\alpha, 1}(z,s) &=  \frac{s^{\alpha-1}}{p(1-r_2(s))r_1(s)}\sum_{n=0}^{\infty}\frac{1}{r_1^n(s)}e^{nz}
	  =  \frac{s^{\alpha-1}}{p(1-r_2(s))r_1(s)}\sum_{n=0}^{\infty}\frac{1}{r_1^n(s)}\sum_{k=0}^\infty \frac{n^kz^k}{k!} \notag\\
	  &=  \frac{s^{\alpha-1}}{p(1-r_2(s))r_1(s)}\sum_{k=0}^{\infty}\Big(\sum_{n=0}^\infty\frac{1}{r_1^n(s)} n^k\Big) \frac{z^k}{k!}.
	 \end{align}
	 Then a coefficient comparison gives the Laplace transform for all  
	 \be
	 \mathcal L\{ \E((L_{\alpha,1}^{(1)}(t) )^k)\}
	 =  \frac{s^{\alpha-1}}{p(1-r_2(s))r_1(s)}\Big(\sum_{n=0}^\infty\frac{1}{r_1^n(s)} n^k\Big) 
	 =   \frac{s^{\alpha-1}}{p(1-r_2(s))r_1(s)} \Phi\Big(\frac{1}{r_1(s)}, -k ,0 \Big).
	 \ee
	 The function $\Phi$ is the Hurwitz-Lerch transcendent. In our case, since $k$ is an integer, it can be computed explicitly 
	 with repeated derivatives of the geometric series. It can be directly verified that for any $a >1$
	 \[
	  \Phi\Big(\frac{1}{a}, -k ,0 \Big)  = - a  \frac{d}{da}\Phi\Big(\frac{1}{a}, -k+1 ,0 \Big), 
	 \]
	which gives a fast way to compute the Laplace transform of the moments. 	  
	We can then directly compute the Laplace transform of the mean and, by inverting it, its formula. We have  
	\begin{align}
	\mathcal L\{\mu_{L^{(1)}_{\alpha,1}(t)}\}(s) &= \frac{2p-1}{s^{1+\alpha}} + \frac{1-p}{s^{\alpha}}\widetilde{p}_{0,0}(s)\\
	&\phantom{xxxxxxi}\Longleftrightarrow \mu_{L^{(1)}_{\alpha,1}(t)} =  \frac{2p-1}{\Gamma(1+a)}t^{\alpha} + \frac{1-p}{\Gamma(a)} \int_{0}^t (t - u)^{\alpha-1}p_{0,0}(u)\,du \label{eq:LenLap} \notag\\
	&\phantom{xxxxxxxxxxxxxxxx} = \frac{2p-1}{\Gamma(1+a)}t^{\alpha} + (1-p) J^{\alpha}p_{0,0}(t).
	\end{align}
	Above we used $J^{\alpha}p_{0,0}(t) $ to denote the Riemann -Liouville fractional integral of $p_{0,0}(t) $, in general defined by
	\[
	J^{\alpha}f(t) = \frac{1}{\Gamma(\alpha)} \int_0^t (y - t)^{\alpha-1}f(y)\,dy. 
	\]
	Similarly, we can compute by hand the Laplace transform for the second moment of the queue length, given by
	\be \label{eq:ohno}
	 \mathcal L\{ \E((L_{\alpha,1}^{(1)}(t) )^2)\} = \widetilde{p}_{0,0}(s) \frac{r_1(s)(1+r_1(s))}{(r_1(s)-1)^3}
	 = \frac{p^3}{s^{3a}} \widetilde{p}_{0,0}(s) r_1(s)(1+r_1(s))(1-r_2(s))^3.
	\ee
	Repeated applications of \eqref{eq:useful} and \eqref{eq:lappoo} can then `simplify' the expression into 
	\begin{align*}
	 \mathcal L\{ \E((L_{\alpha,1}^{(1)}(t) )^2)\} &= \frac{1}{s^{3a+1}}\Big( \theta_1(p)+\theta_2(p)s^{\alpha} 
	 + \theta_3(p)s^{2\alpha} \Big) \\ 
	 &\phantom{xxxxxxxxxxxxxxxxxxx} +  \frac{1}{s^{3a}} \widetilde{p}_{0,0}(s)\Big( \theta_4(p)+\theta_5(p)s^{\alpha} 
	 + \theta_6(p)s^{2\alpha} \Big),
	\end{align*}
	which can the be inverted should a close formula be desirable. 
	
	\begin{example}[Asymptotics for the probability of an empty queue] 
	\label{ex:emptyQ1}
	First assume $p \neq 1/2$. As $s \to 0$ we 
	Taylor expand $\widetilde{p}_{0,0}(s)$
	\begin{align}
		\widetilde{p}_{0,0}(s) & 
		= \frac{1-2p}{2(1-p)}\frac{1}{s}-\frac{1}{2(1-p)s^{1-\alpha}} + \frac{|1-2p|}{2(1-p)s}\sqrt{ 1+ \frac{2}{(1-2p)^2}s^\alpha +o(s^\alpha)}\notag \\
		&= \Big( \frac{1-2p}{2(1-p)}  + \frac{|1-2p|}{2(1-p)} \Big)\frac{1}{s} 
		+ \frac{1}{2(1-p)}\Big(\frac{1}{|1-2p|}  - 1\Big)\frac{1}{s^{1-\alpha}} +o(s^{\alpha-1}). \label{eq:laptayl}
	\end{align} 
	If $p < 1/2$, which means $Q_n$ is positive recurrent, the leading term above is $\frac{1-2p}{1-p}s^{-1}$ as $s \to 0$. 
	Then from the Tauberian theorem of Laplace transforms we have 
	\be \label{eq:convequil}
	\frac{1-2p}{1-p} = \lim_{t \to \infty} \frac{ \displaystyle\int_0^t p_{0,0}(u)\, du}{t} =  \lim_{t \to \infty} p_{0,0}(t),
	\ee
	which coincides with \eqref{eq:equilibriumform}. A similar approach can prove the limiting distribution for $p_{0,n}$, 
	via \eqref{eq:pon}. For more on the equilibrium distribution, see the next section. 
	
	If $p > 1/2$, the $s^{-1}$ disappears from \eqref{eq:laptayl} and the leading term is $\frac{1}{2p-1}s^{\alpha-1}$. Therefore, 
	\[
	\frac{1}{(2p-1)\Gamma(2-\alpha)} = \lim_{t \to \infty} \frac{ \displaystyle\int_0^t p_{0,0}(u)\, du}{t^{1-\alpha}} 
	=  \lim_{t \to \infty} \frac{p_{0,0}(t)}{(1-\alpha)t^{-\alpha}},
	\]
	which gives 
	\[
	p_{0,0}(t) \sim \frac{1}{(2p-1)\Gamma(1 - \alpha)}t^{-\alpha}, \quad t \to \infty.
	\]
	Finally, focus on the particular case of $p = 1/2$. We have for $s \to 0$
	\begin{align*}
	\widetilde{p}_{0,0}(s) = \frac{1}{s} - \frac{r_2(s)}{s} &
	= - \frac{1}{s^{1-\alpha}} +\frac{\sqrt2s^{\alpha/2}\sqrt{1 + s^{\alpha}/2}}{s}\\
	&= - \frac{1}{s^{1-\alpha}} +\frac{\sqrt2s^{\alpha/2}(1+s^{\alpha}/4+ o(s^{\alpha}))}{s}
	= \frac{\sqrt2}{s^{1-\alpha/2}} +o(s^{-1+\alpha/2}).
	\end{align*}
	Then as before, 
	\be
	p_{0,0}(t) \sim \frac{\sqrt2}{\Gamma(1-\alpha/2)} t^{-\alpha/2}, \quad t \to \infty.
	\ee
	\end{example}
	
	\begin{example}[Asymptotics for the mean queue length and variance] 
	\label{ex:emptyQ2}
	From Example \ref{ex:emptyQ1} and equation \eqref{eq:LenLap}, 
	we obtain the Laplace transform asymptotics as $s \to 0$ for the queue length 
	\be
	\mathcal L\{\mu_{L^{(1)}_{\alpha,1}(t)}\}(s) = 
	\begin{cases}\displaystyle
		\frac{p}{1-2p}s^{-1}+o(s^{-1}), &\quad p < 1/2,\vspace{0.3cm}\\
		\vspace{0.3cm}
		\displaystyle\frac{\sqrt2}{2s^{1+\alpha/2}}+o(s^{-1-\alpha/2}), &\quad p = 1/2,\\
		\displaystyle\frac{2p-1}{s^{1+\alpha}} +o(s^{-1-\alpha}), &\quad p > 1/2.
	\end{cases}
	\ee 
	This, via the Tauberian theorem for Laplace transforms, gives the asymptotics 
	\be
	\mu_{L^{(1)}_{\alpha,1}(t)} \sim 
	\begin{cases}\displaystyle
		\frac{p}{1-2p}, &\quad p < 1/2,\vspace{0.3cm}\\
		\displaystyle\frac{\sqrt2}{2\Gamma(1+\alpha/2)}t^{\alpha/2}, &\quad p = 1/2,\vspace{0.3cm}\\
		\displaystyle\frac{2p-1}{\Gamma(1+\alpha)}t^{\alpha}, &\quad p > 1/2.
	\end{cases}
	\ee 
	By Taylor expanding \eqref{eq:ohno} up to 3 terms in order to find the asymptotic of the second moment, we also obtain that 
	\be
	\text{Var}(L^{(1)}_{\alpha,1}(t)) \sim \begin{cases}\displaystyle
		\mathcal C_1, &\quad p < 1/2,\vspace{0.3cm}\\
		\displaystyle\mathcal C_2 t^{\alpha}, &\quad p = 1/2,\vspace{0.3cm}\\
		\displaystyle\mathcal C_3t^{2\alpha}, &\quad p > 1/2.
	\end{cases}
	\ee
	Note that when $p \le 1/2$ the Taylor approximation can only be with the first two terms of the expansion but the
	 case $p>1/2$ requires the third term. 
	 
	 \begin{remark}
	 Note that in the case $p > 1/2$ the orders of both the mean and the variance coincide with the 
	 corresponding orders of the fractional Poisson process. This is yet another indication of the transience
	 of the queue length in this case.
	 \end{remark}
	\end{example}
	
	\subsection{Equilibrium and mixing times in the ergodic regime}

	There are several ways to compute the equilibrium probabilities as 
	$t\to \infty$ when $p < 1/2$. For example, as in Example \ref{ex:emptyQ1}, 
	we can perform the same calculations that led to \eqref{eq:convequil} and 
	obtain in general that the limiting mass function of the queue length coincides with that of 
	$Q_n$ from \eqref{eq:equilibriumform}, namely
	\be\label{eq:timecon} 
	\lim_{t \to \infty} p_{0,k}(t) = \pi_k = \frac{1-2p}{1-p}\Big(\frac{p}{1-p}\Big)^k, \quad k \in \N_0.
	\ee 
	
	 This is just a limiting statement for the probabilities for fixed $k$ and we would
	 like to upgrade the statement in terms of mixing time for the chain, in terms of the 
	 total variation distance. 
	
	\begin{proof}[Proof of Theorem \ref{thm:Q1mixtime}]
	 Our starting point is Corollary 3.1 from \cite{Moller89} which states that 
	 for a spatial birth and death chain, potentially on an infinite space
	 (like $Q_n$) we have the following 
	 estimate for the total variation distance between the invariant distribution 
	 $\{\pi_k\}_{k \in \N_0}$ and $\{\P_{\gamma}\{ Q_n = k\}\}_{k\in \N_0}$. The initial distribution 
	 $\gamma$ for the chain must satisfy a non-degeneracy condition (see conditions in \cite{Moller89}) 
	 which is satisfied if for example $\gamma$ has a finite support. 
	 Using the result, we can find 
	 positive constants $c$ and $ \theta <1$ such that 
	 \be
	 \sup_{A \in \mathcal F} \Big| \pi(A) - \P_{\gamma}\{ Q_n \in A\} \Big| \le c \, \theta^{n}.
	 \ee
	 
	 Note that Corollary 3.1 is stated for continuous times birth and death chains, 
	 but indeed the same follows for discrete times by the concentration 
	 of the Poisson process around its mean. 
	
	We use  this estimate to bound the total variation distance between $p_{\gamma}(t)$ and $\pi$.  
	Condition on the events of the counting process to obtain 
	\begin{align*}
	p_{\gamma,k}(t) =  \sum_{n=0}^{\infty} \P_{\gamma}\{ Q_n = k\}  \P\{N^{(\alpha)}(t)=n \}
	\end{align*}
	and then estimate
	\begin{align*}
	\|p_{\gamma}(t) - \pi \|_{TV} &=  \frac{1}{2} \sum_{k \in \N_0} |p_{\gamma,k}(t) -\pi_k| 
	=  \frac{1}{2} \sum_{k \in \N_0} | \sum_{n=0}^{\infty} \P_{\gamma}\{ Q_n = k\}  \P\{N^{(\alpha)}(t)=n \} -\pi_k(t)|\\
	&\le   \frac{1}{2} \sum_{k \in \N_0}  \sum_{n=0}^{\infty} \big|\P_{\gamma}\{ Q_n = k\} - \pi_k\big|  \P\{N^{(\alpha)}(t)=n \}\\
	&=  \sum_{n=0}^{\infty} \big\|\P_{\gamma}\{ Q_n \in \cdot\} - \pi\big\|_{TV}  \P\{N^{(\alpha)}(t)=n \} 
	\le  \sum_{n=0}^{\infty} c \theta^n \P\{N^{(\alpha)}(t)=n \}\\
	& = c \E(\theta^{N^{(\alpha)}(t)}) = c E_{\alpha, 1}( -(1-\theta)t^{\alpha}). 
	\end{align*}
	
	\begin{remark} The above estimate shows that the total variation distance from the equilibrium decays like 
	$t^{-\alpha}$, the decay rate of the Mittag-Leffler function. Note that the upper bound above is strictly monotone in $t$.  
	\end{remark}
	
	For any $\e>0$ define $T^{c, \theta}_{\e}$ to be 
	\[
	T^{c, \theta}_{\e} = \inf\{ t > 0:  c E_{\alpha, 1}( -(1-\theta)t^{\alpha}) \le \e \}. 
	\]
	Then we conclude that
	\be \label{eq:mixtB1}
	T^{\rm{mix}}_\e \le T^{c, \theta}_{\e}. 
	\ee 
	We can bound $T^{c, \theta}_{\e}$ from above the following way. By monotonicity, $t >   T^{c, \theta}_{\e}$ 
	if and only if   
	$c \E(\theta^{N^{(\alpha)}(t)})< \e$. 
	
	There exists two uniform constant $U_{\alpha, \theta}$, $L_{\alpha, \theta}$  
	and time $t_0$ so that for all $t > t_0$ we have 
	\be \label{eq:tail0}
	 \frac{L _{\alpha, \theta}}{t^{\alpha}} \le c E_{\alpha, 1}( -(1-\theta)t^{\alpha})\le \frac{U_{\alpha, \theta}}{t^{\alpha}}.
	\ee
	This follows directly from the asymptotic behaviour in equation \eqref{eq:asym1} where also a suitable $t_0$ can be selected. 
	When $t =  T^{c, \theta}_{\e}$, we have by continuity that  $c E_{\alpha, 1}( -(1-\theta)( T^{c, \theta}_{\e})^{\alpha})=\e$.
	Let $\e \le L_{\alpha, \theta}t_0^{-\alpha}$ which in tern gives 
	\[
	\e = c E_{\alpha, 1}( -(1-\theta)( T^{c, \theta}_{\e})^{\alpha}) \le  L_{\alpha, \theta}t_0^{-\alpha} \le  c E_{\alpha, 1}( -(1-\theta)t_0^{\alpha}).
	\]
	By the monotonicity of the 
	Mittag-Leffler function we then have that  $T^{c, \theta}_{\e} > t_0$. 
	Thus we are allowed to use \eqref{eq:tail0} and we must have  
	\[
	 \frac{L_{\alpha, \theta}}{(T^{c, \theta}_{\e})^{\alpha}} \le \e \le  \frac{U_{\alpha, \theta}}{(T^{c, \theta}_{\e})^{\alpha}} \Longleftrightarrow T^{c, \theta}_{\e} = \mathcal O(\e^{-1/\alpha}).
	\]
	The theorem follows for $\e_0 = L_{\alpha, \theta}t_0^{-\alpha}$ and $U_{\alpha, \lambda, \mu_0}= U_{\alpha, \theta}$.
	\end{proof}
\bigskip
\section{The renewal queue}
	\label{sec:model2}

Recall $p_{\alpha, \beta}^{ \lambda, \mu}$ from \eqref{eq:pal} is the probability of an arrival when a renewal event occurs. We can easily find some special values of $p_{\alpha, \beta}^{ \lambda, \mu}$. 

First, consider the case where either $\alpha$ or $\beta$ is equal to 1 and $\mu = \lambda$. Then we have 
\begin{align*}
1- p_{1, \beta}^{ \lambda, \lambda} &=1-  p_{1, \beta}^{ 1, 1} = P\{ X_{1,1} > X_{\beta, 1}\} = \int_{0}^{\infty} e^{-t}f_{X_{\beta}}(t)\,dt=\E(e^{-X_{\beta}})= \frac{1}{2}.
\end{align*}
So for this choice of parameters the embedded chain will correspond to a null-recurrent queue. 

Then, if we factor in the scalings, we have 
\[
p_{1, \beta}^{ \lambda, \mu} =\P \{ \mu \lambda^{-1} X_{1} < X_{\beta}\}= \int_{0}^{\infty} (1-e^{-\lambda t/\mu})f_{X_{\beta}}(t)\,dt = \frac{\lambda^\alpha}{\mu^{\alpha} + \lambda^{\alpha}}.						
\]
Note that in this case $p_{1, \beta}^{ \lambda, \mu} < 1/2 \Longleftrightarrow \mu > \lambda$, i.e.\ the embedded chain is positive recurrent precisely when the departure process is faster.  

Similarly, for any $\alpha \in (0,1)$, by symmetry  
\[
p_{\alpha,\alpha}^{\lambda, \lambda} = \frac{1}{2}, \]
and by a straightforward coupling argument, 
\[
p_{\alpha,\alpha}^{\lambda, \mu} < \frac{1}{2} \Longleftrightarrow \mu > \lambda. 
\]
\begin{figure}[h]
\includegraphics[height=8cm]{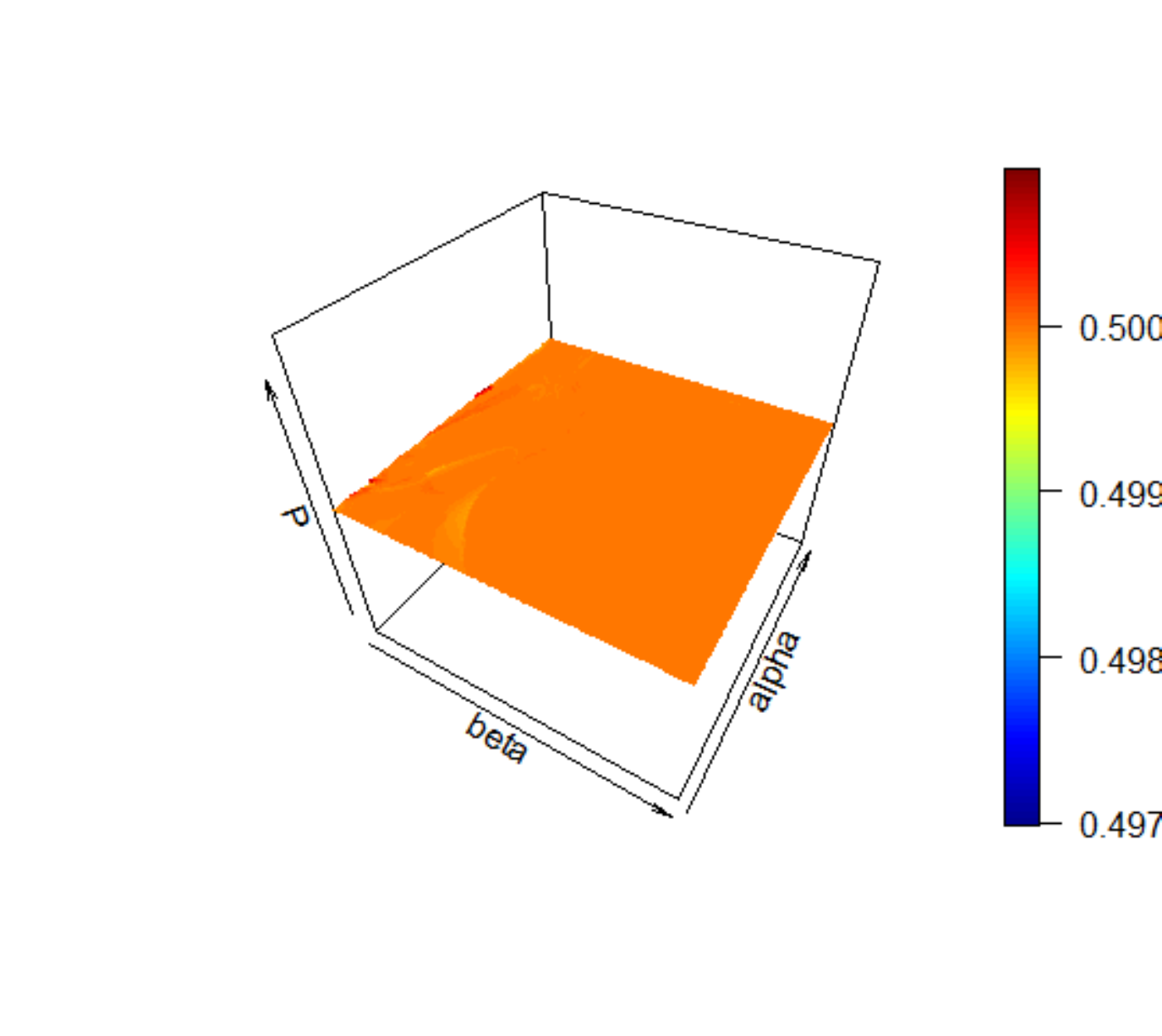}
\caption{A numerical simulation for the value of $p_{\alpha, \beta}^{\lambda, \mu}$ when $\mu=\lambda$. Note that even the minor discrepancies when $\alpha, \beta$ are small are still within $0.001$ distance from 1/2 which is what we expect the true value to be.} 
\label{fig:limits}
\end{figure}

\begin{conjecture} 
\label{lem:1/2}
Let $\alpha$, $\beta$ in $(0,1]$ and let $X_{\alpha}$ and $X_{\beta}$ two independent Mittag-Leffler distributions with power indices $\alpha$, $\beta$ respectively and scalings $\lambda=\mu$. Then 
\[
\P\{X_{\alpha} < X_{\beta}\} = p_{\alpha, \beta}^{\lambda, \lambda}= \frac{1}{2}.
\]
Then, the embedded discrete queue will be null recurrent if and only if $\lambda = \mu$ and positive recurrent if and only $\mu > \lambda$. I.e.\ the behaviour of the queue length depends solely on the time scales, not the power tails.
\end{conjecture}

Unfortunately we have not been able to rigorously prove Conjecture \ref{lem:1/2}, so we have Proposition \ref{prop:renewal} 
Simultaneously, the conjecture is strongly supported by high precision numerical calculations and you can see the results in Figure \ref{fig:limits}.

\begin{proof}[Proof of Proposition \ref{prop:renewal}] 
Since 
\[
p_{\alpha, \beta}^{\lambda, \mu} = p_{\alpha, \beta}^{\lambda/\mu, 1} 
\]
define $\rho = \frac{\lambda}{\mu}$ to be a parameter in $(0,\infty)$, and consider the probabilities  $p_{\alpha, \beta}^{\rho,1} = \P\{ \rho^{-1}X_\alpha <  X_{\beta}\}$. We have 
\[
p_{\alpha, \beta}^{\rho,1} = \int_{0}^\infty \P\{X_{\alpha} \le \rho t\} f_{X_{\beta}}(t)\,dt= \int_{0}^\infty F_{X_{\alpha}} ( \rho t) f_{X_{\beta}}(t)\,dt.
\]
With this representation we see that $p_{\alpha, \beta}^{\rho,1} $ is strictly monotonically increasing (from 0 to 1) as $\rho \to \infty$ by the complete monotonicity of $F_{X_{\alpha}}$ and the dominated convergence theorem. Then, there will be a critical $\rho^*_{\alpha, \beta}$ such that $p_{\alpha, \beta}^{\rho^*_{\alpha, \beta},1}=1/2$.

The proposition then follows, since  $\rho^*_{\alpha, \beta}$ will be the critical $\rho$ that separates positive recurrence and transience in the embedded chain. The embedded chain will be null recurrent when $\rho = \rho^*_{\alpha, \beta}$.
\end{proof}

\begin{remark} Conjecture \ref{lem:1/2} says that $\rho^*_{\alpha, \beta} =1$ for all values of $\alpha, \beta$. This is verified for the cases $\alpha\vee\beta =1$ or $\alpha = \beta$. Note that in this case the characterisation of recurrence and transience of the embedded queue length would be equivalent to that of the M/M/1 queue. \qed
\end{remark}

\section{The Mittag-Leffler GI/GI/1 queue} 

	\label{sec:model3}

	\subsection{Limit theorems} 
	
	To prove these, we will first establish the convergence of the base fractional Poisson process, $N^{(\alpha)}_{\nu}(t)$, for $\alpha \in (0,1)$ and $\nu > 0$.x


\begin{lemma} \label{FLT for single FCP}
Let $\gamma \in (0,1)$. Then we have the following convergence for the fractional Poisson process $N^{(\alpha)}_{\nu}$ in the Skorohod $J_1$ topology dependent on the value of $\gamma$:
\begin{enumerate}
    \item $\gamma = \alpha$ implies that
    \begin{equation}
        \left(\frac{N^{(\alpha)}_{\nu}(t\tau)}{t^{\gamma}}\right)_{\tau \geq 0} \overset{(w)}{\longrightarrow} \left(\nu^{\alpha} Y_{\alpha}(\tau)\right)_{\tau \geq 0},
    \end{equation} \\
    \item $\gamma > \alpha$ implies that
    \begin{equation}
         \left(\frac{N^{(\alpha)}_{\nu}(t\tau)}{t^{\gamma}}\right)_{\tau \geq 0} \overset{(w)}{\longrightarrow} \{0\}_{\tau \geq 0}.
    \end{equation} 
   \end{enumerate}
\end{lemma}

\begin{proof}
The first statement follows from Theorem 6 of \cite{Leonenko19} and equations \eqref{FTPP-FPP relation} - \eqref{FPP self-similarity property} from earlier.

Let us look at the second case, where $\gamma > \alpha$. In this case we can write $\gamma = \alpha + \delta$, with $\delta > 0$. Now, for fixed $t \geq 0$, we can write
\begin{equation}
     \left(\frac{N^{(\alpha)}_{\nu}(t\tau)}{t^{\gamma}}\right)_{\tau \geq 0} \overset{d}{=}  \left(\frac{N^{(\alpha)}_{\nu}(t\tau)}{t^{\alpha}} \cdot \frac{t^{\alpha}}{t^{\alpha + \delta}}\right)_{\tau \geq 0} \overset{d}{=}  \left(\frac{N^{(\alpha)}_{\nu}(t\tau)}{t^{\alpha}} \right)_{\tau \geq 0} \cdot \frac{t^{\alpha}}{t^{\alpha + \delta}}
\end{equation}
It is known that multiplication in the Skorohod $J_1$ topology is continuous as long as the two functions share no common discontinuities. This is true in our case, since for every $t > 0$, $\frac{t^{\alpha}}{t^{\alpha + \delta}}$ can be viewed as a constant function with respect to $\tau$ and contains no discontinuities. It must also follow that the joint distribution of $\left(\frac{N^{(\alpha)}_{\nu}(t\tau)}{t^{\alpha}} \right)_{\tau \geq 0}$ and $\frac{t^{\alpha}}{t^{\alpha + \delta}}$ converges as $t \to \infty$, and we are free to invoke the Continuous Mapping theorem. Since $\frac{t^{\alpha}}{t^{\alpha + \delta}} \to 0$ as $t \to \infty$, it is clear that the whole expression above must converge to the constant function $0$ with respect to the $J_1$ topology. 
\end{proof}

\begin{proof}[Proof of Theorem \ref{thm: functional limit}]
Now that we have established the convergence of the base fractional Poisson process, we can use results from \cite{Whitt80} to push through to the whole queue length expression. First we draw attention to Theorem 4.1 of that paper, which states that the addition of two functions $x,y$ on a Skorohod space $\mathcal{D}$ is a continuous operation provided that $x$ and $y$ share no common discontinuities. Let $x_t$ and $y_t$ correspond to the fractional Poisson processes, $x_t(\tau) = \left\{\frac{N^{(\alpha_1)}_{\lambda}(t \tau)}{t^{\gamma}}\right\}_{\tau \geq 0}$ and $y_t(\tau) = \left\{\frac{N^{(\alpha_2)}_{\mu}(t \tau)}{t^{\gamma}}\right\}_{\tau \geq 0}$. These are two sequences of independent random variables and as such the joint distribution converges as $t \to \infty$. Since these processes share no common discontinuities we can then apply the continuous mapping theorem and Lemma \ref{FLT for single FCP} to show the convergence
\begin{equation}
    \left\{\frac{N^{(\alpha_1)}_{\lambda}(t \tau) - N^{(\alpha_2)}_{\mu}(t \tau)}{t^{\gamma}}\right\}_{\tau \geq 0} \overset{(w)}{\longrightarrow} \begin{cases}
    \left\{\lambda^{\alpha_1} Y_{\gamma}(\tau)\right\}_{\tau \geq 0} &\text{ if } \alpha_1 > \alpha_2, \\
 \left\{- \mu^{\alpha_2} Y_{\gamma}(\tau)\right\}_{\tau \geq 0} &\text{ if } \alpha_1 < \alpha_2, \\
 \left\{\lambda^{\alpha} Y_{\gamma}(\tau) - \mu^{\alpha} \tilde{Y}_{\gamma}(\tau) \right\}_{\tau \geq 0}  &\text{ if } \alpha_1 = \alpha_2 = \alpha.
    \end{cases}
\end{equation}
Now we turn to Theorem 6.4 in \cite{Whitt80}. Setting $c_n, c \equiv 0$ this states that if we have the convergence of $x_t(\tau) \to x(\tau)$ as $t \to \infty$ in the $J_1$ topology, then this implies the convergence $x_t^{\downarrow}(\tau) \to x^{\downarrow}(\tau)$, where
\begin{equation}
    x^{\downarrow}(\tau) = x(\tau) - \inf_{0 \leq s \leq \tau} x(s).
\end{equation}
Combining this with the above result finally leads us to the expression

\begin{align*}
   & \left\{\frac{L^{(\alpha_1, \alpha_2)}_{\lambda, \mu}(t \tau)}{t^{\gamma}}\right\}_{\tau \geq 0} \overset{(w)}{\longrightarrow} \\
   &\phantom{xxxxxxxxxxxx} \begin{cases} \vspace{0.2cm}
    \left\{\lambda^{\alpha_1} Y_{\gamma}(\tau) - \inf_{0 \leq s \leq \tau}\left(\lambda^{\alpha_1} Y_{\gamma}(s)\right)\right\}_{\tau \geq 0} &\text{ if } \alpha_1 > \alpha_2, \\
    \vspace{0.2cm}
 \left\{- \mu^{\alpha_2} Y_{\gamma}(\tau) - \inf_{0 \leq s \leq \tau}\left(-\mu^{\alpha_2} Y_{\gamma}(s)\right)\right\}_{\tau \geq 0} &\text{ if } \alpha_1 < \alpha_2, \\
 \left\{\lambda^{\alpha} Y_{\gamma}(\tau) - \mu^{\alpha} \tilde{Y}_{\gamma}(\tau) - \inf_{0 \leq s \leq \tau}\left(\lambda^{\alpha} Y_{\gamma}(s) - \mu^{\alpha} \tilde{Y}_{\gamma}(s)\right) \right\}_{\tau \geq 0}  &\text{ if } \alpha_1 = \alpha_2 = \alpha.
    \end{cases}
\end{align*}
To return to the original expression from the theorem, it suffices to observe that for all $\tau > 0$
\begin{equation*}
    \inf_{0 \leq s \leq \tau}\left(\lambda Y_{\gamma}(s)\right) = 0,
\end{equation*}
and
\begin{equation*}
    \inf_{0 \leq s \leq \tau}\left(-\mu^{\alpha_2} Y_{\gamma}(s)\right) = -\mu^{\alpha_2} Y_{\gamma}(\tau),
\end{equation*}
by the weak monotonicity of the subordinator.
\end{proof}

	\subsection{Recurrence results} 
	On order to show the recurrence results of Theorems 
	\ref{thm:Rec+Trans}, \ref{thm:Rec+Trans=Balanced}, we need a few preliminary lemmas, 
	as well as a coupling argument that will allow us to use a regeneration argument.  
	Our first lemma works for any pair of counting process $N_{\lambda}(t), N_{\mu}(t)$ 
	and queue length given by 
	\[
   		 L_t = \left(N_{\lambda}(t) - N_{\mu}(t)\right) 
		 - \inf_{0 \le s \le t}\left\{N_{\lambda}(s) - N_{\mu}(s) \right\}.  
	\]
	It offers a way to decide when we have an empty service. 
	
	\begin{lemma} \label{lem:useful}The following are equivalent. 
	\begin{enumerate}
	\item Time $T$ is a discontinuity point of the infimum process, i.e. 
		\be \label{eq:infequiv}
		\inf_{0 \le s \le T}\left\{N_{\lambda}(s) - N_{\mu}(s) \right\} = \inf_{0 \le s \le T-}\left\{N_{\lambda}(s) - N_{\mu}(s)\right\}-1. 
		\ee
	\item  The departure process $N_{\mu}(t)$ has a renewal point at time $T$ and queue length satisfies 
		\[
		L_{T-}=L_{T}=0.
		\] 	
	\end{enumerate}	
		In other words, at time $T$ there was an unused service time if and only if \eqref{eq:infequiv} is satisfied.
	\end{lemma}

	\begin{proof} We prove the direct implication. 
	Since the infimum decreases by $1$ at time $T$ by the assumption, we must have that 
	$N_{\mu}(T)$ increased by 1 at time $T$, i.e. $T$ is a renewal point of $N_{\mu}$. Then we have that 
	\begin{align*}
		0 \le L_{T-} &= \left(N_{\lambda}(T-) - N_{\mu}(T-)\right) 
		 - \inf_{0 \le s \le T-}\left\{N_{\lambda}(s) - N_{\mu}(s) \right\}\\
		 &=  N_{\lambda}(T) - (N_{\mu}(T)-1) - \left( \inf_{0 \le s \le T}\left\{N_{\lambda}(s) - N_{\mu}(s) \right\} +1\right)\\
		 &=  N_{\lambda}(T) - N_{\mu}(T) - \inf_{0 \le s \le T}\left\{N_{\lambda}(s) - N_{\mu}(s) \right\} = L_T\le (L_{T-}-1)\vee0.
	\end{align*}
	The last inequality holds because we know that the departure process has an event. 
	But then, for the whole string of inequalities to hold, it must be that $L_T- = 0$. 
	
	The converse implication is left to the reader, but the approach can be found 
	in the proof of \eqref{two queue relation} in Appendix \ref{sec:A}.
	\end{proof}

	\begin{proof}[Proof of Theorem \ref{thm:Rec+Trans}] 
	
	First consider the case $\alpha_1 < \alpha_2$. The proof is independent of the time scalings 
	$\lambda$ and $\mu$; what is important is the tail exponent.  
	
	We begin with the following estimate which will be valid for any $s>0$.
		\begin{align*}
			\P\Big\{ \inf_{0 \le s \le t}\{N^{(\alpha_1)}(s) &- N^{(\alpha_2)}(s)\} > -t^{\alpha_1/2} \Big\}\le \P\big\{N^{(\alpha_1)}(t) - N^{(\alpha_2)}(t) > -t^{\alpha_1/2} \big\}\\
			&\le e^{st^{\alpha_1/2}}\E(e^{sN^{(\alpha_1)}(t)})\E(e^{-sN^{(\alpha_2)}(t)}) \text{ via a Chernoff bound,}\\
			&= e^{st^{\alpha_1/2}}E_{\alpha_1,1}((e^{s} -1)t^{\alpha_1})E_{\alpha_2,1}((e^{-s} - 1)t^{\alpha_2}).
		\end{align*} 
		At this point we make a particular choice for $s$ since we can choose any positive value for it. 
		To have a bound that tends to 0 as $t \to \infty$, set $s = t^{-\alpha_1}$. Then we obtain 
		\begin{align*}
		\P\Big\{ \inf_{0 \le s \le t}\{N^{(\alpha_1)}(t) &- N^{(\alpha_2)}(t)\} > -t^{\alpha_1/2} \Big\}\\
		&\le e^{t^{-\alpha_1/2}}E_{\alpha_1,1}((e^{t^{-\alpha_1}} -1)t^{\alpha_1})E_{\alpha_2,1}((e^{-t^{-\alpha_1}} - 1)t^{\alpha_2})\\
		&= e^{t^{-\alpha_1/2}} E_{\alpha_1,1}( 1+ O(-t^{-\alpha_1}))E_{\alpha_2,1}( -t^{\alpha_2 - \alpha_1}+ o(t^{\alpha_2 - \alpha_1})) \\
		&\le C \frac{1}{t^{\alpha_2-\alpha_1}}, \quad \text{ for $t$ large enough}.
		\end{align*}
	To see the last inequality, observe that as $t$ grows, $e^{t^{-\alpha_1/2}}$ converges to 1 and $E_{\alpha_1,1}( 1+ O(-t^{-\alpha_1}))$ converges by continuity to $E_{\alpha_1,1}( 1))$. Therefore both of these terms are bounded by some constant $C_1$ for large $t$. The third term is the one who dictates the behaviour at infinity. For $t$ large enough, and by the monotonicity of $E_{\alpha_2,1}$ we have
	\[ 
	E_{\alpha_2,1}( -t^{\alpha_2 - \alpha_1}+ o(t^{\alpha_2 - \alpha_1}))  \le E_{\alpha_2,1}\Big( -\frac{1}{2}\Big(t^{\frac{\alpha_2 - \alpha_1}{\alpha_2}}\Big)^{\alpha_2}\Big) \le \frac{C_2}{\Big(t^{\frac{\alpha_2 - \alpha_1}{\alpha_2}}\Big)^{\alpha_2}}= \frac{C_2}{t^{\alpha_2 - \alpha_1}}.
	\]
	Then, define the events  
	\[
	\mathcal A_n = \Big\{ \inf_{0 \le s \le n}\{N^{(\alpha_1)}(s) - N^{(\alpha_2)}(s)\} \le -n^{\alpha_1/2} \Big\}, \quad \P\{\mathcal A_n\} \ge 1 - \frac{C}{n^{\alpha_2 - \alpha_1}}. 
	\]	
	We compute 
	\begin{align*}
		\P\{ \mathcal A_n \text{ i.o.} \} &= \P\Big\{ \bigcap_{k=1}^{\infty} \bigcup_{n=k}^{\infty}\mathcal A_n  \Big\} = \lim_{k\to \infty} \P\Big\{\bigcup_{n=k}^{\infty}\mathcal A_n  \Big\} \ge \lim_{k\to \infty} \P\Big\{\mathcal A_k \Big\}\\
		&\ge \lim_{k\to \infty} \Big(1 - \frac{C}{k^{\alpha_2 - \alpha_1}} \Big) =1. 
	\end{align*}
	Therefore, on a full probability event, we can find a sequence of integers $\{ k_j \}_{j \ge 1}$ so that 
	\[ 
	f(j) = \inf_{0 \le s \le k_j }\{N^{(\alpha_1)}(s) - N^{(\alpha_2)}(s)\} 
	\]
	is a strictly decreasing integer function. As such, by Lemma \ref{lem:useful} 
	we know that we can find an infinite sequence of times $t_j$
	 for which the departure process had an unused service time and that happens 
	 if and only if the queue length was already zero. Therefore the queue length becomes zero 
	 infinitely often with probability 1. 
	
	For the case $\alpha_1 > \alpha_2$ we reason similarly, but for $s <0$.
	\begin{align*}
		 \P\big\{N^{(\alpha_1)}(t) - N^{(\alpha_2)}(t) < t^{\alpha_2/2} \big\}
			&\le e^{-st^{\alpha_2/2}}\E(e^{sN^{(\alpha_1)}(t)})\E(e^{-sN^{(\alpha_2)}(t)}) \text{ via a Chernoff bound,}\\
			&= e^{-st^{\alpha_2/2}}E_{\alpha_1,1}((e^{s} -1)t^{\alpha_1})E_{\alpha_2,1}((e^{-s} - 1)t^{\alpha_2}).
	\end{align*} 
		At this point we make a particular choice for $s$ since we can choose any positive value for it. 
		To have a bound that tends to 0 as $t \to \infty$, set $s = -t^{-\alpha_2}$. Then we obtain 
		\begin{align*}
		\P\Big\{N^{(\alpha_1)}(t) &- N^{(\alpha_2)}(t)  < t^{\alpha_2/2} \Big\}\\
		&\le e^{t^{-\alpha_2/2}}E_{\alpha_2,1}((e^{t^{-\alpha_2}} -1)t^{\alpha_2})E_{\alpha_1,1}((e^{-t^{-\alpha_2}} - 1)t^{\alpha_1})\\
		&\le e^{t^{-\alpha_2/2}} E_{\alpha_2,1}( 1+ O(-t^{-\alpha_2}))E_{\alpha_1,1}( -t^{\alpha_1 - \alpha_2}+ o(t^{\alpha_1 - \alpha_2})) \\
		&\le C \frac{1}{t^{\alpha_1-\alpha_2}}, \quad \text{ for $t$ large enough}.
		\end{align*}
		
		In order to conclude, use the fact that 
		$
		L^{\alpha_1, \alpha_2} (t) \ge N^{\alpha_1}(t) - N^{\alpha_2}(t),
		$
		which gives 
		\[
		\P\{ L^{\alpha_1, \alpha_2} (t) \ge t^{\alpha/2} \} \ge 1-  \frac{C}{t^{\alpha_1-\alpha_2}}.
		\]
		Finally, for any $n \in \N$, define 
		\[\mathcal B_n = \{  L^{\alpha_1, \alpha_2} (n) \ge n^{\alpha/2} \}\] and notice that 
		\be
		\P\{ \varlimsup_{t\to \infty} L^{\alpha_1, \alpha_2}(t) = \infty \} \ge \P\{ \mathcal B_n \text { i.o. }\}
		\ge \lim_{n\to \infty} \Big(1 - \frac{C}{n^{\alpha_1 - \alpha_2}}\Big) = 1. \qedhere
		\ee
	\end{proof}
	
	\subsection{The proof of Theorem \ref{thm:Rec+Trans=Balanced}}
	The situation in Theorem \ref{thm:Rec+Trans=Balanced} requires a slightly more delicate approach 
	and two coupling arguments which we present first. We begin with a construction of a `sped-up' 
	queue which also provides a regeneration time for the dynamics.
	
	Let $N_a(t)$ denote the arrival counting process and $N_d(t)$ denote the departure counting process. 
	Let $L_t$ denote the corresponding queue length. Assume that at time $T$ we have a departure event, 
	and assume that the next arrival event is at time $T+ \eta$. 
	Define new processes 
	\[
	\widetilde N_{a}(t) = 
	\begin{cases}
	N_a(t), & t < T-,\\ 
	N_{a}(\eta+t), & t \ge T
	\end{cases}
	\]
	which just speeds up the arrival process by $\eta$. Then we define a new queue length at time $T$
	\be \label{eq:coupleq}
	\widetilde L_t = 
	\begin{cases}
	L_t, & t \le T-\\
	1+ \widetilde N_{a}(t) - N_{d}(t)
		 - \inf_{0 \le s \le t}\left\{\widetilde N_{a}(s) - N_{d}(s) \right\}, & t \ge T.
	\end{cases}
	\ee	
	This just says that at time $T$ the departure event happened before 
	we brought forward the next departure event. Note that $\widetilde L_{t-T}, t \ge T$ 
	has the same distribution as $L_t$ with an initial condition of $\widetilde L_T = L_T+1$. 
	
	\begin{lemma}\label{lem:coupling}
	Consider the queue $\widetilde L_t$ defined in \eqref{eq:coupleq}. 
	Let $T_0$ be the departure time event after which we use $\widetilde N_{a}(t)$ and 
	let $T_1$ be the first unused service time after $T_0$ for the process $L_t$. 
	(a priori $T_1$ could be infinity, and that is also fine). Then, 
	\be\label{eq:coupineq}
	L_t \le   \widetilde L_t, \quad \text{ for all } t < T_1.
	\ee
	\end{lemma}
	
	Intuitively the lemma is clear, since we are just speeding up the arrival process by the random value 
	$\eta$, however one should acknowledge the possibility that by doing so, a
	 customer who appeared earlier may have benefited for a previously unused service event, 
	 thus making the second queue length shorter. Indeed by the construction 
	 this will  not happen until the first unused departure time
	 and this is shown  in the next proof. See also Figure \ref{fig:coupledQ} for a pictorial explanation.  
	 
	 \begin{figure}[h]
	 \begin{center}

\tikzset{every picture/.style={line width=0.75pt}} 

\begin{tikzpicture}[x=0.75pt,y=0.75pt,yscale=-1,xscale=1]

\draw    (39,2089.9) -- (618,2089.9) ;
\draw    (40,2133.76) -- (619,2133.76) ;
\draw    (44,2272.94) -- (623,2272.94) ;
\draw  [dash pattern={on 4.5pt off 4.5pt}]  (87.5,2089.48) -- (87.5,2272.1) ;
\draw  [dash pattern={on 4.5pt off 4.5pt}]  (420.5,2135.02) -- (420.5,2272.1) ;
\draw  [dash pattern={on 4.5pt off 4.5pt}]  (329.5,2133.76) -- (329.5,2240.89) ;
\draw  [dash pattern={on 4.5pt off 4.5pt}]  (314.5,2090.32) -- (314.5,2240.04) ;
\draw  [dash pattern={on 4.5pt off 4.5pt}]  (290,2074.71) -- (290,2271.25) ;
\draw  [dash pattern={on 4.5pt off 4.5pt}]  (260.5,2134.18) -- (262,2271.25) ;
\draw  [dash pattern={on 4.5pt off 4.5pt}]  (206.5,2134.18) -- (206.5,2240.04) ;
\draw  [dash pattern={on 4.5pt off 4.5pt}]  (170.5,2089.48) -- (170.5,2240.04) ;
\draw    (206.5,2074.71) -- (206.5,2148.94) ;
\draw  [dash pattern={on 4.5pt off 4.5pt}]  (564.5,2090.32) -- (564.5,2271.25) ;
\draw  [dash pattern={on 4.5pt off 4.5pt}]  (511.5,2135.02) -- (511.5,2272.1) ;
\draw  [dash pattern={on 4.5pt off 4.5pt}]  (466.5,2135.02) -- (466.5,2272.94) ;
\draw  [fill={rgb, 255:red, 74; green, 144; blue, 226 }  ,fill opacity=1 ] (254,2134.18) .. controls (254,2131.15) and (256.91,2128.7) .. (260.5,2128.7) .. controls (264.09,2128.7) and (267,2131.15) .. (267,2134.18) .. controls (267,2137.21) and (264.09,2139.66) .. (260.5,2139.66) .. controls (256.91,2139.66) and (254,2137.21) .. (254,2134.18) -- cycle ;
\draw  [fill={rgb, 255:red, 245; green, 166; blue, 35 }  ,fill opacity=1 ] (81,2089.48) .. controls (81,2086.45) and (83.91,2083.99) .. (87.5,2083.99) .. controls (91.09,2083.99) and (94,2086.45) .. (94,2089.48) .. controls (94,2092.5) and (91.09,2094.96) .. (87.5,2094.96) .. controls (83.91,2094.96) and (81,2092.5) .. (81,2089.48) -- cycle ;
\draw  [fill={rgb, 255:red, 245; green, 166; blue, 35 }  ,fill opacity=1 ] (558,2090.32) .. controls (558,2087.29) and (560.91,2084.84) .. (564.5,2084.84) .. controls (568.09,2084.84) and (571,2087.29) .. (571,2090.32) .. controls (571,2093.35) and (568.09,2095.8) .. (564.5,2095.8) .. controls (560.91,2095.8) and (558,2093.35) .. (558,2090.32) -- cycle ;
\draw  [fill={rgb, 255:red, 74; green, 144; blue, 226 }  ,fill opacity=1 ] (505,2135.02) .. controls (505,2132) and (507.91,2129.54) .. (511.5,2129.54) .. controls (515.09,2129.54) and (518,2132) .. (518,2135.02) .. controls (518,2138.05) and (515.09,2140.51) .. (511.5,2140.51) .. controls (507.91,2140.51) and (505,2138.05) .. (505,2135.02) -- cycle ;
\draw  [fill={rgb, 255:red, 74; green, 144; blue, 226 }  ,fill opacity=1 ] (460,2135.02) .. controls (460,2132) and (462.91,2129.54) .. (466.5,2129.54) .. controls (470.09,2129.54) and (473,2132) .. (473,2135.02) .. controls (473,2138.05) and (470.09,2140.51) .. (466.5,2140.51) .. controls (462.91,2140.51) and (460,2138.05) .. (460,2135.02) -- cycle ;
\draw  [fill={rgb, 255:red, 74; green, 144; blue, 226 }  ,fill opacity=1 ] (414,2135.02) .. controls (414,2132) and (416.91,2129.54) .. (420.5,2129.54) .. controls (424.09,2129.54) and (427,2132) .. (427,2135.02) .. controls (427,2138.05) and (424.09,2140.51) .. (420.5,2140.51) .. controls (416.91,2140.51) and (414,2138.05) .. (414,2135.02) -- cycle ;
\draw  [fill={rgb, 255:red, 245; green, 166; blue, 35 }  ,fill opacity=1 ] (308,2090.32) .. controls (308,2087.29) and (310.91,2084.84) .. (314.5,2084.84) .. controls (318.09,2084.84) and (321,2087.29) .. (321,2090.32) .. controls (321,2093.35) and (318.09,2095.8) .. (314.5,2095.8) .. controls (310.91,2095.8) and (308,2093.35) .. (308,2090.32) -- cycle ;
\draw  [fill={rgb, 255:red, 74; green, 144; blue, 226 }  ,fill opacity=1 ] (323,2133.76) .. controls (323,2130.73) and (325.91,2128.28) .. (329.5,2128.28) .. controls (333.09,2128.28) and (336,2130.73) .. (336,2133.76) .. controls (336,2136.79) and (333.09,2139.24) .. (329.5,2139.24) .. controls (325.91,2139.24) and (323,2136.79) .. (323,2133.76) -- cycle ;
\draw  [fill={rgb, 255:red, 245; green, 166; blue, 35 }  ,fill opacity=1 ] (284,2090.32) .. controls (284,2087.29) and (286.91,2084.84) .. (290.5,2084.84) .. controls (294.09,2084.84) and (297,2087.29) .. (297,2090.32) .. controls (297,2093.35) and (294.09,2095.8) .. (290.5,2095.8) .. controls (286.91,2095.8) and (284,2093.35) .. (284,2090.32) -- cycle ;
\draw  [fill={rgb, 255:red, 245; green, 166; blue, 35 }  ,fill opacity=1 ] (164,2089.48) .. controls (164,2086.45) and (166.91,2083.99) .. (170.5,2083.99) .. controls (174.09,2083.99) and (177,2086.45) .. (177,2089.48) .. controls (177,2092.5) and (174.09,2094.96) .. (170.5,2094.96) .. controls (166.91,2094.96) and (164,2092.5) .. (164,2089.48) -- cycle ;
\draw  [fill={rgb, 255:red, 74; green, 144; blue, 226 }  ,fill opacity=1 ] (200,2134.18) .. controls (200,2131.15) and (202.91,2128.7) .. (206.5,2128.7) .. controls (210.09,2128.7) and (213,2131.15) .. (213,2134.18) .. controls (213,2137.21) and (210.09,2139.66) .. (206.5,2139.66) .. controls (202.91,2139.66) and (200,2137.21) .. (200,2134.18) -- cycle ;
\draw [color={rgb, 255:red, 189; green, 16; blue, 224 }  ,draw opacity=1 ][line width=2.25]    (44,2272.94) -- (90,2272.1) ;
\draw [color={rgb, 255:red, 189; green, 16; blue, 224 }  ,draw opacity=1 ][line width=2.25]    (317,2214.74) -- (330,2214.74) ;
\draw [color={rgb, 255:red, 189; green, 16; blue, 224 }  ,draw opacity=1 ][line width=2.25]    (291,2240.04) -- (316,2240.04) ;
\draw [color={rgb, 255:red, 189; green, 16; blue, 224 }  ,draw opacity=1 ][line width=2.25]    (262,2271.25) -- (290,2271.25) ;
\draw [color={rgb, 255:red, 189; green, 16; blue, 224 }  ,draw opacity=1 ][line width=2.25]    (206.5,2240.04) -- (261,2240.04) ;
\draw [color={rgb, 255:red, 189; green, 16; blue, 224 }  ,draw opacity=1 ][line width=2.25]    (171,2214.74) -- (206,2214.74) ;
\draw [color={rgb, 255:red, 189; green, 16; blue, 224 }  ,draw opacity=1 ][line width=2.25]    (88,2240.04) -- (170.5,2240.04) ;
\draw [color={rgb, 255:red, 189; green, 16; blue, 224 }  ,draw opacity=1 ][line width=2.25]    (564,2240.89) -- (623,2240.89) ;
\draw [color={rgb, 255:red, 189; green, 16; blue, 224 }  ,draw opacity=1 ][line width=2.25]    (422,2272.1) -- (567,2272.1) ;
\draw [color={rgb, 255:red, 189; green, 16; blue, 224 }  ,draw opacity=1 ][line width=2.25]    (333,2240.89) -- (422,2240.89) ;
\draw    (207,2074.71) -- (289,2074.71) ;
\draw    (43,2358.9) -- (622,2358.9) ;
\draw    (44,2402.76) -- (623,2402.76) ;
\draw    (48,2541.94) -- (627,2541.94) ;
\draw  [dash pattern={on 4.5pt off 4.5pt}]  (91.5,2358.48) -- (91.5,2541.1) ;
\draw  [dash pattern={on 4.5pt off 4.5pt}]  (424.5,2404.02) -- (424.5,2541.1) ;
\draw  [dash pattern={on 4.5pt off 4.5pt}]  (333.5,2402.76) -- (333.5,2509.89) ;
\draw  [dash pattern={on 4.5pt off 4.5pt}]  (235.5,2359.32) -- (235.5,2483.74) ;
\draw  [dash pattern={on 4.5pt off 4.5pt}]  (211.5,2359.32) -- (211.5,2483.74) ;
\draw  [dash pattern={on 4.5pt off 4.5pt}]  (264.5,2403.18) -- (266,2540.25) ;
\draw  [dash pattern={on 4.5pt off 4.5pt}]  (174.5,2358.48) -- (174.5,2509.04) ;
\draw  [dash pattern={on 4.5pt off 4.5pt}]  (210.5,2358) -- (210.5,2417.94) ;
\draw  [dash pattern={on 4.5pt off 4.5pt}]  (485.5,2359.32) -- (485.5,2540.25) ;
\draw  [dash pattern={on 4.5pt off 4.5pt}]  (515.5,2404.02) -- (515.5,2541.1) ;
\draw  [dash pattern={on 4.5pt off 4.5pt}]  (470.5,2404.02) -- (470.5,2541.94) ;
\draw  [fill={rgb, 255:red, 74; green, 144; blue, 226 }  ,fill opacity=1 ] (258,2403.18) .. controls (258,2400.15) and (260.91,2397.7) .. (264.5,2397.7) .. controls (268.09,2397.7) and (271,2400.15) .. (271,2403.18) .. controls (271,2406.21) and (268.09,2408.66) .. (264.5,2408.66) .. controls (260.91,2408.66) and (258,2406.21) .. (258,2403.18) -- cycle ;
\draw  [fill={rgb, 255:red, 245; green, 166; blue, 35 }  ,fill opacity=1 ] (85,2358.48) .. controls (85,2355.45) and (87.91,2352.99) .. (91.5,2352.99) .. controls (95.09,2352.99) and (98,2355.45) .. (98,2358.48) .. controls (98,2361.5) and (95.09,2363.96) .. (91.5,2363.96) .. controls (87.91,2363.96) and (85,2361.5) .. (85,2358.48) -- cycle ;
\draw  [fill={rgb, 255:red, 245; green, 166; blue, 35 }  ,fill opacity=1 ] (479,2359.32) .. controls (479,2356.29) and (481.91,2353.84) .. (485.5,2353.84) .. controls (489.09,2353.84) and (492,2356.29) .. (492,2359.32) .. controls (492,2362.35) and (489.09,2364.8) .. (485.5,2364.8) .. controls (481.91,2364.8) and (479,2362.35) .. (479,2359.32) -- cycle ;
\draw  [fill={rgb, 255:red, 74; green, 144; blue, 226 }  ,fill opacity=1 ] (509,2404.02) .. controls (509,2401) and (511.91,2398.54) .. (515.5,2398.54) .. controls (519.09,2398.54) and (522,2401) .. (522,2404.02) .. controls (522,2407.05) and (519.09,2409.51) .. (515.5,2409.51) .. controls (511.91,2409.51) and (509,2407.05) .. (509,2404.02) -- cycle ;
\draw  [fill={rgb, 255:red, 74; green, 144; blue, 226 }  ,fill opacity=1 ] (464,2404.02) .. controls (464,2401) and (466.91,2398.54) .. (470.5,2398.54) .. controls (474.09,2398.54) and (477,2401) .. (477,2404.02) .. controls (477,2407.05) and (474.09,2409.51) .. (470.5,2409.51) .. controls (466.91,2409.51) and (464,2407.05) .. (464,2404.02) -- cycle ;
\draw  [fill={rgb, 255:red, 74; green, 144; blue, 226 }  ,fill opacity=1 ] (418,2404.02) .. controls (418,2401) and (420.91,2398.54) .. (424.5,2398.54) .. controls (428.09,2398.54) and (431,2401) .. (431,2404.02) .. controls (431,2407.05) and (428.09,2409.51) .. (424.5,2409.51) .. controls (420.91,2409.51) and (418,2407.05) .. (418,2404.02) -- cycle ;
\draw  [fill={rgb, 255:red, 245; green, 166; blue, 35 }  ,fill opacity=1 ] (229,2359.32) .. controls (229,2356.29) and (231.91,2353.84) .. (235.5,2353.84) .. controls (239.09,2353.84) and (242,2356.29) .. (242,2359.32) .. controls (242,2362.35) and (239.09,2364.8) .. (235.5,2364.8) .. controls (231.91,2364.8) and (229,2362.35) .. (229,2359.32) -- cycle ;
\draw  [fill={rgb, 255:red, 74; green, 144; blue, 226 }  ,fill opacity=1 ] (327,2402.76) .. controls (327,2399.73) and (329.91,2397.28) .. (333.5,2397.28) .. controls (337.09,2397.28) and (340,2399.73) .. (340,2402.76) .. controls (340,2405.79) and (337.09,2408.24) .. (333.5,2408.24) .. controls (329.91,2408.24) and (327,2405.79) .. (327,2402.76) -- cycle ;
\draw  [fill={rgb, 255:red, 245; green, 166; blue, 35 }  ,fill opacity=1 ] (205,2359.32) .. controls (205,2356.29) and (207.91,2353.84) .. (211.5,2353.84) .. controls (215.09,2353.84) and (218,2356.29) .. (218,2359.32) .. controls (218,2362.35) and (215.09,2364.8) .. (211.5,2364.8) .. controls (207.91,2364.8) and (205,2362.35) .. (205,2359.32) -- cycle ;
\draw  [fill={rgb, 255:red, 245; green, 166; blue, 35 }  ,fill opacity=1 ] (168,2358.48) .. controls (168,2355.45) and (170.91,2352.99) .. (174.5,2352.99) .. controls (178.09,2352.99) and (181,2355.45) .. (181,2358.48) .. controls (181,2361.5) and (178.09,2363.96) .. (174.5,2363.96) .. controls (170.91,2363.96) and (168,2361.5) .. (168,2358.48) -- cycle ;
\draw  [fill={rgb, 255:red, 74; green, 144; blue, 226 }  ,fill opacity=1 ] (204,2403.18) .. controls (204,2400.15) and (206.91,2397.7) .. (210.5,2397.7) .. controls (214.09,2397.7) and (217,2400.15) .. (217,2403.18) .. controls (217,2406.21) and (214.09,2408.66) .. (210.5,2408.66) .. controls (206.91,2408.66) and (204,2406.21) .. (204,2403.18) -- cycle ;
\draw [color={rgb, 255:red, 189; green, 16; blue, 224 }  ,draw opacity=1 ][line width=2.25]    (48,2541.94) -- (94,2541.1) ;
\draw [color={rgb, 255:red, 189; green, 16; blue, 224 }  ,draw opacity=1 ][line width=2.25]    (515.5,2541.1) -- (627,2541.94) ;
\draw [color={rgb, 255:red, 189; green, 16; blue, 224 }  ,draw opacity=1 ][line width=2.25]    (487.5,2509.74) -- (512.5,2509.74) ;
\draw [color={rgb, 255:red, 189; green, 16; blue, 224 }  ,draw opacity=1 ][line width=2.25]    (264,2483.25) -- (332,2483.25) ;
\draw [color={rgb, 255:red, 189; green, 16; blue, 224 }  ,draw opacity=1 ][line width=2.25]    (210,2483.74) -- (236,2483.74) ;
\draw [color={rgb, 255:red, 189; green, 16; blue, 224 }  ,draw opacity=1 ][line width=2.25]    (175,2483.74) -- (210,2483.74) ;
\draw [color={rgb, 255:red, 189; green, 16; blue, 224 }  ,draw opacity=1 ][line width=2.25]    (92,2509.04) -- (174.5,2509.04) ;
\draw [color={rgb, 255:red, 189; green, 16; blue, 224 }  ,draw opacity=1 ][line width=2.25]    (424.5,2541.94) -- (485.5,2541.94) ;
\draw [color={rgb, 255:red, 189; green, 16; blue, 224 }  ,draw opacity=1 ][line width=2.25]    (337,2509.89) -- (426,2509.89) ;
\draw [color={rgb, 255:red, 189; green, 16; blue, 224 }  ,draw opacity=1 ][line width=2.25]    (237,2452.74) -- (263,2452.74) ;

\draw (12,2076.21) node [anchor=north west][inner sep=0.75pt]    {$N_{a}$};
\draw (13,2119.23) node [anchor=north west][inner sep=0.75pt]    {$N_{d}$};
\draw (13,2258.56) node [anchor=north west][inner sep=0.75pt]    {$L$};
\draw (189,2137.1) node [anchor=north west][inner sep=0.75pt]    {$T_0$};
\draw (443,2137.91) node [anchor=north west][inner sep=0.75pt]    {$ \begin{array}{l}
T_{1}\\
\end{array}$};
\draw (245,2049.44) node [anchor=north west][inner sep=0.75pt]    {$\eta $};
\draw (16,2345.21) node [anchor=north west][inner sep=0.75pt]    {$\widetilde{N_{a}}$};
\draw (17,2388.23) node [anchor=north west][inner sep=0.75pt]    {$N_{d}$};
\draw (17,2527.56) node [anchor=north west][inner sep=0.75pt]    {$L$};
\draw (193,2406.1) node [anchor=north west][inner sep=0.75pt]    {$T_0$};
\draw (447,2406.91) node [anchor=north west][inner sep=0.75pt]    {$ \begin{array}{l}
T_{1}\\
\end{array}$};

\end{tikzpicture}
	\end{center}
	\caption{Schematic for the proof of Lemma \ref{lem:coupling}. Time $T_0$ is a departure time (up until that point the two queue lengths are equal). $\eta$ is the time until next arrival in the original arrival process $N_a$ which becomes $0$ in the sped up process $\tilde N_a$. $T_1$ is the first unused service after $T$. Then up to $T_1$ the queue length of the sped up process is not smaller than the original one. As the picture suggests, it is possible for this ordering to persist for longer, but that depends on the realisation of the two processes and it may break down (as is also shown in the picture). }
	\label{fig:coupledQ}	
	\end{figure}
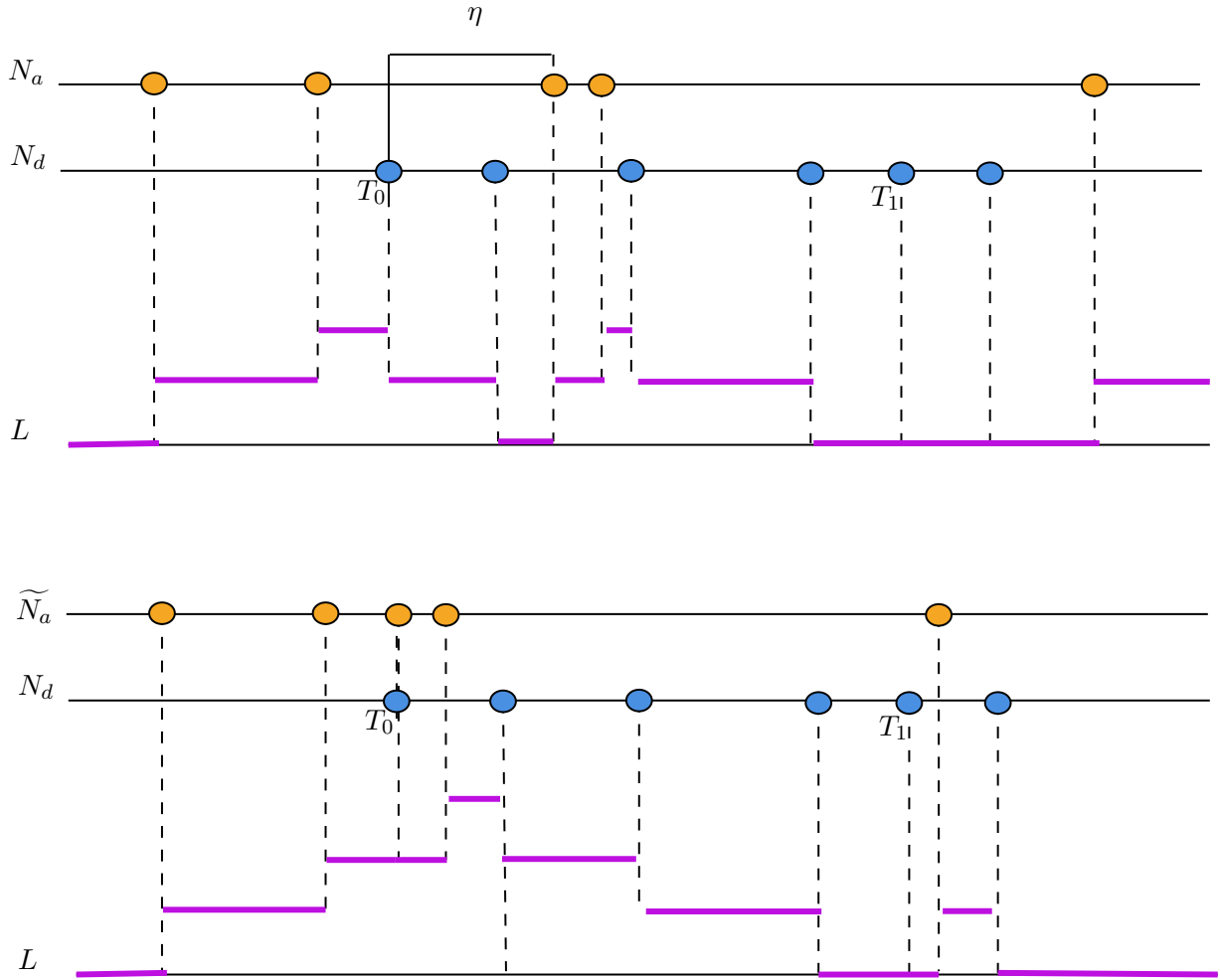
	
	\begin{proof}[Proof of Lemma \ref{lem:coupling}]
		Since no modification happens up to time $T_0$, we have that 
		$L_t =  \widetilde L_t$ for all $t < T_0$. At time $T_0$ we have a departure event. 
		If $L_{T_0-} > 0$, by construction we have that 
		\[ L_{T_0} = L_{T_0-}-1 = \widetilde L_{T_0} - 1 < \widetilde L_{T_0}. \]
		Similarly, if $L_{T_0-} = 0$, the departure at $T_0$ would have been unused and 
		$L_{T_0} = 0$ while in the modified queue $\widetilde L_{T_0} = 1$. 
		
		We now have the claim at the initial time $T_0$ and we proceed by 
		induction on the departure events and keep track of both queue lengths $L_t,  \widetilde L_t$. 
		Let $t_1 < t_2 < \ldots$ be departure events, $t_1 > T_0$. Assume that in 
		$[T_0, t_1]$ the original arrival process  had $k_1\ge 0$ arrivals. 
		That would make the queue length $L_{t_1} = (L_{T_0} + k_1 -1)\vee 0$. 
		The sped-up process $\widetilde N_{a}$ rang $k_1$ times (with the first ring now at $T_0$)
		 and because of the shift by $\eta$, it may have rang even more times. Therefore, up to time $t_1$,  
		 $L_{t_1} \le \widetilde L_{t_1}$ since only one departure occurred. 
		 
		 Now assume that up to time $t_n$ the claim remains true and $L_{t_n} \le \widetilde L_{t_n}$ 
		 while no unused service time for $L_{t_n}$ occurred. 
		 We see what happens at time $t_{n+1}$.   
		 
		  We have the following inequalities 
		 \[
		 k_{n+1} = N_a(t_{n+1}) - N_a(T_0-) \le \widetilde N_a(t_{n+1}) - \widetilde N_a(T_0-)= \tilde k_{n+1}. 
		 \] 
		 
		 Assume that  $t_{n+1}$ is not an unused service time for $L_t$, 
		 otherwise we have nothing to show. This also implies that no unused service times occurred 
		 for $\widetilde L_{t}$ since it was always no less than $L_t$. Then
		 \[
		 \widetilde L_{t_{n+1}} = L_{T_0-}+1 + \tilde k_{n+1} - n-1 =  L_{T_0-} + \tilde k_{n+1} - n \ge  L_{T_0-} +  k_{n+1} - n = L_{t_{n+1}},
		 \]		 
as required. 
	\end{proof}
	
	\begin{lemma}\label{lem:infio}
	Assume $\alpha_1 = \alpha_2=\alpha$ and $\mu = \lambda$. 
	Let  
	$\{ X^{(i)}_{\alpha, \lambda} \}_{i \ge 1}$, $\{ Y^{(i)}_{\alpha, \lambda} \}_{i \ge 1}$  be the inter-event times
	of the arrival and departure process respectively that are i.i.d.\ Mittag-Leffler distributed. 
	Consider the random walk 
	\[
	S_n = \sum_{i=1}^n (X^{(i)}_{\alpha, \lambda} - Y^{(i)}_{\alpha, \lambda}), \quad S_0 = 0
	\]	
	Then
	\begin{enumerate}
		\item $S_n$ will change sign infinitely often, $\displaystyle \varlimsup_{n\to \infty} S_n = +\infty, \varliminf_{n\to\infty} S_n = -\infty$.
		\item Assume that $S_n < 0$ and $S_{n+1} > 0$. Then there exists a time $t < t_{n+1}$  for which 
		 the two processes rang exactly $n$ times and the next event in the process is a departure attempt.
	\end{enumerate}
	\end{lemma}
	
	\begin{proof}[Proof of Lemma \ref{lem:infio}] 
		Define 
		\[
		S^X_n = \sum_{i=1}^n X^{(i)}_{\alpha, \lambda} , \quad S^Y_n = \sum_{i=1}^n  Y^{(i)}_{\alpha, \lambda}.
		\]
		Each of these represent the time of the $n$-th arrival and the time of the $n$-th departure attempt respectively. 
		\begin{enumerate}
			\item Since $X^{(i)}_{\alpha, \lambda} - Y^{(i)}_{\alpha, \lambda}$ is a symmetric random variable, the random walk is an oscillating random walk. The claim then follows from Feller (See Section 12.2 in \cite{Feller08}).
			\medskip
			
			\item  
			Define
			\be \label{tn+1}
			t_{n+1} = S^Y_{n+1}= \sum_{i=1}^{n+1} Y^{(i)}_{\alpha, \lambda}
			\ee
			
			At this time the departure process has its $n+1$-th event. Since $S_n  < 0$ we have that  the $n$-th arrival occurring at time $S^X_n$ happened before (or with) the $n$-th (and therefore before the $n+1$) departure attempt. The `with' here is in general a measure 0 event, except at the beginning when we have renewal point for both processes. Since $S_{n+1} > 0$ we also have that $n+1$ arrival happens after $t_{n+1}$. So by time $t_{n+1}$, both processes jumped $n$ times, and at time $t_{n+1}$ we have the $n+1$ departure attempt.  \qedhere
		\end{enumerate}
	\end{proof}

%
%
	
	\begin{proof}[Proof of Theorem \ref{thm:Rec+Trans=Balanced}]
	
	We first treat the case $\alpha_1 = \alpha_2= \alpha$ and $\mu = \lambda$. 
	To begin with we show that we can find a time $T_1$ for which $L_{T_1}=0$ with probability 1.
	
	Assume the queue begins at $L_0 = 1$, with both arrival and departure renewals beginning at 0 and let 
	the position of the random walks 
	\[
		S^X_n = \sum_{i=1}^n X^{(i)}_{\alpha, \lambda} , \quad S^Y_n = \sum_{i=1}^n  Y^{(i)}_{\alpha, \lambda}.
	\]
	denote the time of the $n$-th arrival and departure respectively. First consider the first sign change 
	of their difference from negative to positive which will occur infinitely often by Lemma \ref{lem:infio}. 
	This occurs as we said above when both processes rang the same amount of times. 
	At time $t_{n+1}$ (from \eqref{tn+1}) we have the $n+1$ departure event. At that point, we either have a queue of length 1 
	and this will be a service time (since both processes rang $n$ times) making  $L_{t_{n+1}}=0$ or, 
	if the queue has length $L_{t_{n+1}-} \ge 2$  we have a departure. 
	But if the queue had length more than 1, then we must have had an unused departure event before, 
	say at time $t_k$ for some $k \le n$. This can only happen if the queue was already empty at time $t_k$.
	 In either case, we have found a  departure time $T_1 > 0$ at which the queue was empty. 
	 Departure time $T_1$ exists with probability 1. 
			
	Define 
		\[
		T_1 = \inf\big\{t >0: N^{Y}_{\alpha, \lambda}(t) \,\text{ rings }, L^{\alpha,\lambda}_{t-} = 0\big\}.
		\]
			At time $T_1$ (an event that clears the queue) 
			we apply Lemma \ref{lem:coupling}. Let $\eta >0$ be such that $T_1 + \eta$ 
			is the time of the next arrival event, and speed up the arrival process by $\eta$. 
			Per the construction in Lemma 	\ref{lem:coupling}, the new queue length $\tilde L$ 
			starts at 1 and we have renewal events at the beginning making the new length equal to 1. 
			But then, either $\tilde L_t$ will hit 0 with probability 1 if the original process does not have an
			 unused service and therefore the original process hits 0 by the inequality 
			 in Lemma \ref{lem:coupling}, 
			or the original process will have an unused service which also means it hit 0. 
			Therefore, in either case, the original process will hit 0 at a time $T_2 > T_1$. 
			Iterate this argument ad infinitum to see that the queue will be empty infinitely often. 
			This proves the case $\mu = \lambda$. 
			
			When $\mu > \lambda$ we can write that the departure events occur faster since 
			\be \label{eq:coup}
			S^{Y,\,\mu}_n = \sum_{i=1}^n  Y^{(i)}_{\alpha, \mu} \stackrel{(d)}{=} \frac{\lambda}{\mu} \sum_{i=1}^n  Y^{(i)}_{\alpha, \lambda}< S^{Y,\,\lambda}_n. 
			\ee
			Therefore, the queue empties faster (in distribution). In terms of the difference random walk, 
			it tends to be positive and tending to infinity suggesting a lot of unused services. 
			As such the queue hits 0 infinitely often. The coupling argument utilises \eqref{eq:coup}  and goes by arguing that by the time $S^{Y,\,\lambda}_n$ has an unused time, then already at a prior time $S^{Y,\,\mu}_n$ had an unused service time, and then Lemma \ref{lem:coupling} is utilised at the time the queue hit 0. The details are left to the interested reader. 
	
		Finally, to prove the second part of the theorem for $\lambda \ge \mu$.
		 We can repeat the arguments in the proof of Theorem \ref{thm:Rec+Trans} but now the scalings 
		$\lambda$ and $\mu$ play a role. 
		We have $N^{(\alpha)}_{\lambda}(t)$ denote the arrival renewal process 
		and $N^{(\alpha)}_{\mu}(t)$ denote the departure process with Mittag-Leffler $\alpha$ parameter 
		and scalings  $\lambda, \mu$ respectively.  For any $s<0$

	\begin{align*}
		 \P\big\{N^{(\alpha)}_{\lambda}(t) - N^{(\alpha)}_{\mu}(t) < t^{\alpha/2} \big\}
			&\le e^{-st^{\alpha/2}}\E(e^{sN^{(\alpha)}_{\mu}(t)})\E(e^{-sN^{(\alpha)}_{\lambda}(t)}) \text{ via a Chernoff bound,}\\
			&= e^{st^{\alpha/2}}E_{\alpha,1}((e^{s} -1)\lambda^\alpha t^{\alpha})E_{\alpha,1}((e^{-s} - 1)\mu^{\alpha}t^{\alpha}).
	\end{align*} 
		At this point we make a particular choice for $s$ since we can choose any negative value for it. 
		Select $s =- c^{\alpha} t^{-\alpha}$. Then the bound above becomes 
		\[
		\P\Big\{N^{(\alpha)}_\lambda(s) - N^{(\alpha)}_{\mu}(s) < t^{\alpha/2} \Big\} \le 
		 e^{c^{\alpha}t^{-\alpha/2}}E_{\alpha,1}(-(c \lambda)^\alpha +o(t^{-\alpha}))E_{\alpha,1}((c \mu)^{\alpha}+ o(t^{-\alpha}))
		\]
		Then we can find $T_0(\mu, \lambda), c_0(\mu, \lambda) > 0$ such that for  $t > T_0$ and $0< c < c_0$ so that the Taylor expansion at $c = 0$ of the upper bound above becomes 
		\[
		\Big(1 + \frac{c^\alpha}{t^{\alpha/2}} + \e_1\Big)\Big( 1 - \frac{(c\lambda)^\alpha}{\Gamma(1+\alpha)}+\e_2\Big) \Big(1+\frac{(c\mu)^\alpha}{\Gamma(1+\alpha)}+\e_3\Big) = 1 + c^{\alpha} \frac{\mu^{\alpha} - \lambda^\alpha}{\Gamma(1+\alpha)} + o (c^{\alpha}) \vee o (t^{-\alpha/2}) < 1.
		\]
	The last inequality follows only when $\mu > \lambda$. Errors $\e_1, \e_2$ and $\e_3$ 
	are all of order $o (c^{\alpha}) \vee o (t^{-\alpha/2})$.
	In order to conclude, use the fact that 
		$
		L^{\alpha_1, \alpha_2} (t) \ge N^{\alpha_1}(t) - N^{\alpha_2}(t).
		$
	The rest of the proof proceeds as in Theorem \ref{thm:Rec+Trans}, 
	where we obtain that  
		\[
		\P\{ \varlimsup_{t \to \infty} L^{\alpha, \alpha}_{\lambda, \mu}  = + \infty \} \ge 1 - c_0^{\alpha} \frac{\lambda^\alpha - \mu^\alpha}{\Gamma(1+\alpha)} > 0.
		\]

		Similarly, if $\mu = \lambda$ we obtain that 
		\[
		\P\{ \varlimsup L^{\alpha, \alpha}_{\lambda, \lambda} = + \infty \} \ge 1- c^{2\alpha} \frac{\lambda^{2\alpha}}{\Gamma(1+\alpha)^2} > 0. \qedhere
		\]
		\end{proof}.

\appendix

\section{Inductive proof of equation  \eqref{two queue relation}} 
\label{sec:A}

\begin{proof}[ Proof of equation \eqref{two queue relation}]
Let $N_a(t)$ and $N_d(t)$ denote the two counting processes representing the arrivals and departures attempts respectively and assume that $N_a(0) = N_d(0) =0$. 

We will prove the result by induction over all jump events in the time interval $[0,t]$, of which by assumption there are at most finitely many. Order all the jump times in the interval $[0,t]$ and label them using $0< \tau_1 < \tau_2 < \ldots < \tau_M \le t$.

The queue length remains at 0 until time $\tau_1-$, since no process jumped. At time $\tau_1$ we have a jump. If the jump is an arrival, then $N_a(\tau_1)=1$ while still $N_d(\tau_1)=0$.  The difference $N_a(\tau_1) - N_d(\tau_1) = 1$ while the infimum of the difference is still attained at $t = 0$ and it equals 0. So the right-hand side equals 1, which equals the queue length since the queue experienced an arrival. In the case where the jump was for departures, we have that 
\[
N_a(\tau_1) - N_d(\tau_1) = -1 = \inf_{0 \le s\le \tau_1}\{ N_a(s) - N_d(s)\},
\]
where the infimum is attained precisely at $\tau_1$. 
If $\tau_1 = \tau_2$ and both processes jump, arrivals happen before departures, so  $L_{\tau_1} = 0$ and $N_a(\tau_1) - N_d(\tau_1)=0$. 

This forms the base case of the induction. Now assume that the equation holds up to event $\tau_k$. Up until $t < \tau_{k+1}$ no process jumps, so the equality is maintained up to $\tau_{k+1}-$. 
At $\tau_{k+1}$ at least one process jumps. 

If it is $N_a$, the infimum does not change since its argument increased at this time. So both $L_{\tau_{k+1}}$ and $N_a(\tau_{k+1}) - N_d(\tau_{k+1})$ increased by 1 and equality is preserved. 

If the process that jumped is $N_d$, there are two cases to consider. If $L_{\tau_{k+1}-} = 0$  then this does not change with a jump of $N_d$. For the right-hand side 
\[
N_a(\tau_{k+1}-) - N_d(\tau_{k+1}-)= \inf_{0 \le s < \tau_{k+1}}\{ N_a(s) - N_d(s)\}.
\]
In this situation, a jump of $N_d$ actually reduces the infimum by 1 since 
\begin{align*} 
\inf_{0 \le s \le \tau_{k+1}}\{ N_a(s) - N_d(s)\} &= \inf_{0 \le s < \tau_{k+1}}\{ N_a(s) - N_d(s)\}  \wedge (N_a(\tau_{k+1}) - N_d(\tau_{k+1}))\\
&= N_a(\tau_{k+1}-) - N_d(\tau_{k+1}-) -1,
\end{align*}
which also shows the right hand side is 0. 

The other option is that $L_{\tau_{k+1}-} > 0$. Then 
\[
N_a(\tau_{k+1}-) - N_d(\tau_{k+1}-) -1\ge  \inf_{0 \le s < \tau_{k+1}}\{ N_a(s) - N_d(s)\}
\]
which leads to
\begin{align*} 
\inf_{0 \le s \le \tau_{k+1}}\{ N_a(s) - N_d(s)\} &= \inf_{0 \le s < \tau_{k+1}}\{ N_a(s) - N_d(s)\}  \wedge (N_a(\tau_{k+1}) - N_d(\tau_{k+1}))\\
&= \inf_{0 \le s < \tau_{k+1}}\{ N_a(s) - N_d(s)\}  \wedge (N_a(\tau_{k+1}-) - N_d(\tau_{k+1}-)-1)\\
&=\inf_{0 \le s < \tau_{k+1}}\{ N_a(s) - N_d(s)\}. 
\end{align*}
In other words, in this case the infimum remains unchanged, while $N_a(\tau_{k+1}) - N_d(\tau_{k+1}) = N_a(\tau_{k+1}-) - N_d(\tau_{k+1}-) -1$ and the right hand-side is reduced by 1, as is the queue length.

The case of the simultaneous jump is treated similarly. 
\end{proof}

\begin{remark}
While in the article we are dealing with continuous processes with the probability of both jumping at the same time equal to 0, it could be that other models might exhibit simultaneous jumps. In that case, the formula still remains true if we adopt the convention that arrivals happen before departures (instantaneously). It wouldn't necessarily hold if departures were happening before arrivals. The problem would come from events where the queue length process is already at 0:   
\[
N_a(\tau_{k+1}-) - N_d(\tau_{k+1}-)= \inf_{0 \le s < \tau_{k+1}}\{ N_a(s) - N_d(s)\}.
\]
Then if at $\tau_{k+1}$ we have a simultaneous jump and thus
\[
N_a(\tau_{k+1}-) - N_d(\tau_{k+1}-) = N_a(\tau_{k+1}) - N_d(\tau_{k+1}),
\]
while
\begin{align*} 
\inf_{0 \le s \le \tau_{k+1}}\{ N_a(s) - N_d(s)\} &= \inf_{0 \le s < \tau_{k+1}}\{ N_a(s) - N_d(s)\}  \wedge (N_a(\tau_{k+1}) - N_d(\tau_{k+1}))\\
&= N_a(\tau_{k+1}-) - N_d(\tau_{k+1}-)= N_a(\tau_{k+1}) - N_d(\tau_{k+1}).
\end{align*}
Therefore if the formula remained true, the queue length would have been zero. However, if arrivals are happening after departures we have that queue length would equal 1. \qed
\end{remark}
%

\section{Proof of Theorem \ref{thm:Q1mgf}}
\label{sec:B}
\begin{proof}[Proof of Theorem \ref{thm:Q1mgf}]
	Let $M_{\alpha, \lambda}(z,t) = \E(e^{z L^{(1)}_{\alpha, \lambda}(t)}) $ denote the moment generating function 
	of the queue length at time $t$. For simplicity we assume $\lambda =1$; the only effect of a 
	different $\lambda$ below would be that the right hand side is multiplied with an extra $\lambda^\alpha$. 
	\allowdisplaybreaks
	\begin{align*}
	\frac{d^{\alpha}}{dt^{\alpha}}M_{\alpha, 1}(z,t) 
	&= \frac{d^{\alpha}}{dt^{\alpha}}\E(e^{z L^{(1)}_{\alpha, 1}(t)})
		= \sum_{n=0}^{\infty}e^{zn}\frac{d^{\alpha}}{dt^{\alpha}}p_{0,n}(t)\\ 
	&= \frac{d^{\alpha}}{dt^{\alpha}}p_{0,0}(t)  +  \sum_{n=1}^{\infty}e^{zn}( - p_{0,n}(t) +  p p_{0,n-1}(t) + (1-p) p_{0,n+1}(t))\\
	&= \frac{d^{\alpha}}{dt^{\alpha}}p_{0,0}(t) - (M_{\alpha, 1}(z,t) -p_{0,0}(t))\\
	&\phantom{xxxxx}+ p  e^z M_{\alpha, 1}(z,t) +(1-p)e^{-z}( M_{\alpha, 1}(z,t) -  p_{0,0}(t) -  e^zp_{0,1}(t)))\\
	&= - p p_{0,0}(t) + (1-p) p_{0,1}(t) + p_{0,0}(t) - (1-p)e^{-z}p_{0,0}(t) - (1-p)p_{0,1}(t)\\
	&\phantom{xxxxx} + \big(pe^z  - 1  + (1-p)e^{-z}\big)M_{\alpha, 1}(z,t) \\
	&=(1-p)(1 - e^{-z}) p_{0,0}(t) +  \big(pe^z  - 1  + (1-p)e^{-z}\big)M_{\alpha, 1}(z,t).
	\end{align*}
	Multiply through by $e^z$ to obtain 
	\be
	e^z\frac{d^{\alpha}}{dt^{\alpha}}M_{\alpha, 1}(z,t) 
	= (1-p)(e^{z}-1)p_{0,0}(t) + \big(pe^{2z} - e^z + (1-p)\big)M_{\alpha, 1}(z,t).
	\ee
	We compute the Laplace transform (for the $t$ variable) and we see 
	\be
	e^z\Big( s^{\alpha} \widetilde{M}_{\alpha, 1}(z,s) - s^{\alpha-1}\Big) 
	=  (1-p)(e^{z}-1)\widetilde{p}_{0,0}(s) + \big(pe^{2z} - e^z + (1-p)\big)\widetilde{M}_{\alpha, 1}(z,s).
	\ee	
	Solving for  $\widetilde{M}_{\alpha, 1}(z,s)$ we obtain
	\begin{align} \label{lapmoment}
	\widetilde{M}_{\alpha, 1}(z,s) &= \frac{e^zs^{\alpha-1} 
	+  (1-p)(e^{z}-1)\widetilde{p}_{0,0}(s)}{e^z s^{\alpha} - \big(pe^{2z} - e^z + (1-p)\big)}.
	\end{align}
	
	Set $e^z = u$. The roots of the denominator  $r_1(s), r_2(s)$, with $r_1(s)> r_2(s)$ for the $u$ variable 
	are given by 
	\be\label{eq:momroots}
	r_{1,2}(s) = \frac{1+s^{\alpha} \pm \sqrt{(1 - 2p)^2 + 2s^{\alpha} + s^{2\alpha}}}{2p}
	=  \frac{1+s^{\alpha} \pm \sqrt{(1 + s^{\alpha})^2 -4p(1-p)}}{2p}.
	\ee
	The following relations hold, and we will be using them without particular mention
	\be \label{eq:useful}
	r_1+r_2 = \frac{1+s^{\alpha}}{p}, \quad r_1r_2 = \frac{1-p}{p}, \quad (1-r_1)(1-r_2)=\frac{-s^a}{p}.
	\ee
	
	The first expression \eqref{eq:momroots} shows that the roots are real numbers, 
	while the second shows that $r_2(s) >0$.
	Moreover it is straightforward to check that $r_2(s) < 1$ while the third equality 
	in \eqref{eq:useful} gives $r_1(s)>1$.  
	
	The expression for $\widetilde{p}_{0,0}$ can be calculated using Rouche's Theorem for the probability generating function \cite{Bailey54}. Briefly, equation \ref{lapmoment} can be converted to an equivalent equation for the p.g.f again using the substitution $e^z = u$; then Rouche's Theorem can be applied to show that the denominator contains a root inside the unit disc, which must match with a root of the numerator. This was done in \cite{Cahoy15}, giving
	 \[
	 r_2(s)s^{a-1} + (1-p)(r_2(s)-1) \widetilde{p}_{0,0}(s) =0 \Longleftrightarrow 
	 \widetilde{p}_{0,0}(s)= \frac{r_2(s)s^{\alpha-1}}{(1-p) (1- r_2(s))}.
	 \]
	  In particular, we obtain the Laplace transform 
	  \be\label{eq:lappoo}
	  \widetilde{p}_{0,0}(s) = \frac{s^{a-1}}{p(1-r_2(s))r_1(s)} 
	  = \frac{r_2(s)s^{a-1}}{(1-p)(1-r_2(s))} = \frac{1}{s} - \frac{p}{1-p} \frac{r_2(s)}{s}.
	  \ee
	  \begin{remark} In the case $\alpha = 1$ the inverse Laplace transform can be given in terms of elementary 
	  functions and we obtain
	  the formulas for the M/M/1 queue given in \cite{Bailey54} The paper gives results for the p.g.f.~ but one can check for the m.g.f.~ by making the substitution $u=e^z$.  
	  \end{remark}
	  Substitute  $\widetilde{p}_{0,0}(s)$ in \eqref{lapmoment} to see
	  \[
	  \widetilde{M}_{\alpha, 1}(z,s) 
	  =\frac{-s^{\alpha-1}}{p(e^z - r_1(s))(1-r_2(s))},
	  \]
	  which concludes the proof of the theorem.
	  \end{proof}

\section{Derivation of the fractional evolution equations}
\label{sec:C}

Let $q_{i,j}$ denote the transition probabilities of moving from $i$ customers to $j$ customers at an event time. Then we can say
\begin{align} \label{Transition probilities}
    q_{k,k} &= \beta, \\
    q_{k,k+1} &= \begin{cases}
        1 - \beta \hspace{15.pt} &\text{ if } k = 0 \\
        (1-\beta) \cdot p \hspace{15.pt} &\text{ o/w,}
    \end{cases} \\
    q_{k, k-1} &= \begin{cases}
        0 \hspace{15.pt} &\text{ if } k = 0 \\
        (1-\beta) \cdot (1-p) \hspace{15.pt} &\text{ o/w,}
    \end{cases}
\end{align}
with all other probabilities equal to zero.

Now define $p_{i}(t) = \bP\{X(t) = i \}$ to be the probability of finding $i$ customers in the queue at time $t$. We then have that $p_i(t)$ satisfies the forward equations
\begin{equation}
    p_i(t) = \bar{F}_T^{(\alpha)}(t)\delta_0 + \sum_{k \in \{i - 1, i, i + 1\}}q_{k, i}\int^t_0 p_k(u)f^{(\alpha)}_T(t-u) du.
\end{equation}
We now want to Laplace transform equation to get
\begin{equation}
    \tilde{p_i}(s) = \tilde{\bar{F}}_T^{(\alpha)}(s)\delta_0 + \beta\tilde{f}^{(\alpha)}_T(s)\tilde{p}_i(s) + \left((1-\beta)p\right)\tilde{f}^{(\alpha)}_T(s)\tilde{p}_{i-1}(s) + \left((1-\beta)(1-p)\right)\tilde{f}^{(\alpha)}_T(s)\tilde{p}_{i+1}(s)
\end{equation}
Now let us multiply both sides of this equation through by $s$ and then subtract $\delta_0$ from both sides, which yields
\begin{align}
    \mathcal{L}\left(\frac{dp_i(t)}{dt}; s\right) &= s\tilde{\bar{F}}_T^{(\alpha)}(s)\delta_0 - \delta_0 + s\beta\tilde{f}^{(\alpha)}_T(s)\tilde{p}_i(s) + s\left((1-\beta)p\right)\tilde{f}^{(\alpha)}_T(s)\tilde{p}_{i-1}(s) \\
    &\phantom{xxxx}+ s\left((1-\beta)(1-p)\right)\tilde{f}^{(\alpha)}_T(s)\tilde{p}_{i+1}(s) \nonumber \\
    &= \frac{s^{\alpha}}{\lambda^{\alpha} + s^{\alpha}}\delta_0 - \delta_0 + \frac{s\lambda^{\alpha} \beta}{\lambda^{\alpha} + s^{\alpha}}\tilde{p}_i(s) + \frac{s\lambda^{\alpha}}{\lambda^{\alpha} + s^{\alpha}}\left((1-\beta)p\right)\tilde{p}_{i-1}(s) \nonumber \\
    &\phantom{xxxx}+ \frac{s \lambda^{\alpha}}{\lambda^{\alpha} + s^{\alpha}}\left((1-\beta)(1-p)\right)\tilde{p}_{i+1}(s). \nonumber
\end{align}
Rearranging this gives
\begin{align}
    \frac{\lambda^{\alpha} + s^{\alpha}}{\lambda^{\alpha} s}\mathcal{L}\left(\frac{dp_i(t)}{dt}; s\right) &= -\frac{1}{s}\delta_0 +  \beta\tilde{p}_i(s) + \left((1-\beta)p\right)\tilde{p}_{i-1}(s) 
    + (1-\beta)(1-p)\tilde{p}_{i+1}(s).  
\end{align}
The inverse Laplace transform of $\frac{y + s^{\alpha}}{s}$ is given by
\begin{equation}
    \mathcal{L}^{-1}\left(\frac{y + s^{\alpha}}{s}; t\right) = \frac{t^{-\alpha}}{\Gamma(1-\alpha)} + y,
\end{equation}
and by substituting this in with $y = \lambda^{\alpha}$ we obtain
\begin{align} \label{fractional kolmogorov eqn}
    \frac{d^{\alpha}p_i(t)}{dt^{\alpha}} = \lambda^{\alpha}\left(\beta - 1\right)p_i(t) +  \lambda^{\alpha}(1-\beta)p p_{i-1}(t) + \lambda^{\alpha}(1-\beta)(1-p) p_{i+1}(t). 
\end{align}
Similarly, we can find the boundary conditions
\begin{align} \label{fractional kolmogorov boundary eqn}
    \frac{d^{\alpha}p_0(t)}{dt^{\alpha}} = -  \lambda^{\alpha}(1-\beta)pp_0(t) +  \lambda^{\alpha}(1-\beta)(1-p) p_{1}(t).
\end{align}

\bibliographystyle{plain}
\bibliography{GG1Ref}

\end{document}